\documentclass[12pt]{amsart}
\usepackage{amsmath, amsthm, amssymb, tikz, verbatim}
\usepackage[hidelinks]{hyperref}
\usepackage[all,cmtip]{xy}
\usetikzlibrary{calc,through,backgrounds,snakes,decorations.pathmorphing}
\usepackage[tikz]{bclogo}

\theoremstyle{plain}
\newtheorem{thm}[equation]{Theorem}    \newtheorem{prop}[equation]{Proposition}
\newtheorem{lem}[equation]{Lemma}      \newtheorem{cor}[equation]{Corollary}
\theoremstyle{definition}
\newtheorem{dfn}[equation]{Definition} \newtheorem{exa}[equation]{Example}
\newtheorem{rem}[equation]{Remark}     \newtheorem{obs}[equation]{Observation}

\numberwithin{equation}{section}

\def\a{\alpha}          \def\b{\beta}           \def\g{\gamma}
\def\d{\delta}          \def\e{\epsilon}        \def\la{\lambda}
          \def\D{\Delta}

\def\C{\mathbb{C}}      \def\N{\mathbb{N}}
\def\R{\mathbb{R}}      
\def\rL{\mathrm{L}}     \def\rLS{\mathrm{LS}}   \def\rS{\mathrm{S}}
\def\cB{\mathcal{B}}    \def\cG{\mathcal{G}}    \def\cU{\mathcal{U}}

\def\bo{\bar{o}}        \def\bu{\bar{u}}        \def\bv{\bar{v}}
\def\bx{\bar{x}}        \def\by{\bar{y}}        \def\bz{\bar{z}}
\def\bB{\bar{B}}        \def\bT{\overline{T}\,} \def\tg{\tilde{\g}}

\def\mf{\phi}            \def\mm{\mu}           \def\mml{\mu'}
\def\me{\eta}                  \def\ms{\sigma}
\def\ii{i_{\mathrm{I}}}  \def\il{i_{\mathrm{L}}}\def\is{i_{\mathrm{S}}}

\def\Bq#1#2{\operatorname{\mathcal{#1B#2}}(X)}
\def\Bqo#1#2#3{\operatorname{\mathcal{#1B#2}}_{#3}(X)}
\def\BqO#1#2#3{\operatorname{\mathcal{#1B#2}}^{#3}(X)}
\def\nbq{\Bq{}{}}            \def\lbq{\Bq{L}{}}
\def\nbqs{\Bq{}{S}}          \def\lbqs{\Bq{L}{S}}
\def\lbqso#1{\Bqo{L}{S}{#1}} \def\gbq{\Bq{G}{}}
\def\Gbq#1{\BqO{G}{}{#1}}    \def\sbq{\mathcal{B}_{\text{std}}(X)}
\def\hrcat(#1){rCAT$(#1;*)$} 
\def\rcat(#1){rCAT$(#1)$}    

\def\ebd#1{\partial_{\mathrm{E}}#1}
\def\ibd#1{\partial_{\mathrm{I}}#1}
\def\Gbd#1{\partial_{\mathrm{G}}#1}
\def\bbd#1{\partial_{\mathrm{B}}#1}
\def\lbbd#1{\partial_{\mathrm{LB}}#1}
\def\bsbd#1{\partial_{\mathrm{BS}}#1}
\def\lbsbd#1{\partial_{\mathrm{LBS}}#1}
\def\bcl#1{\overline{#1}_\mathrm{B}}
\def\Gcl#1{\overline{#1}_\mathrm{G}}
\def\icl#1{\overline{#1}_\mathrm{I}}
\def\bt{\tau_{\mathrm{B}}} 
\def\ct{\tau_{\mathrm{C}}} 
\def\Gt{\tau_{\mathrm{G}}} 

\def\c{c}
\def\sym{{\sim}} 
\def\ds{\displaystyle}
\DeclareMathOperator{\len}{len}%
\DeclareMathOperator{\dist}{dist}%
\def\mino{\wedge}    \def\({\left(}
\def\maxo{\vee}      \def\){\right)}
\def\ip#1#2#3{\left<#1,#2;\,#3\right>}
\def\IP<#1,#2;#3>{\ip{#1}{#2}{#3}}

\begin{document}
\title{The Boundary at Infinity of a rough CAT(0) space}

\author{S.M. Buckley and K. Falk}%
\address{Department of Mathematics and Statistics, National University of
Ireland Maynooth, Maynooth, Co. Kildare, Ireland}%
\email{stephen.buckley@maths.nuim.ie}%
\address{Universit\"at Bremen, FB 3 - Mathematik, Bibliothekstra{\ss}e 1,
28359 Bremen, Germany}%
\email{khf@math.uni-bremen.de}%

\begin{abstract}
We develop the boundary theory of rough CAT(0) spaces, a class of length
spaces that contains both Gromov hyperbolic length spaces and CAT(0) spaces.
The resulting theory generalizes the common features of the Gromov boundary
of a Gromov hyperbolic length space and the ideal boundary of a complete
CAT(0) space. It is not assumed that the spaces are geodesic or proper.
\end{abstract}

\subjclass[2010]{Primary 51M05, 51M10. Secondary: 51F99}%
\keywords{CAT(0) space, Gromov hyperbolic space, rough CAT(0) space, ideal
boundary, Gromov boundary, bouquet boundary}
\thanks{The authors were partially supported by Science Foundation Ireland.
Both authors thank the University of Bremen and National University of
Ireland Maynooth for hospitality and financial support of reciprocal visits
that enabled this research effort.}

\maketitle

\section{Introduction}

The boundary theory of Gromov hyperbolic and complete CAT(0) spaces share
common features; by ``boundary'', we always mean some sort of boundary at
infinity. In particular if $X$ is both a Gromov hyperbolic space and a
complete CAT(0) space, then it is well known that its Gromov boundary $\Gbd
X$ and its ideal boundary $\ibd X$ can naturally be identified. Furthermore
under this identification, the canonical topology $\Gt$ generated by the
canonical gauge of metrics on the Gromov boundary equals the cone topology
$\ct$ on the ideal boundary. See Section~\ref{S:prelims} for relevant
definitions and references.

However it would be preferable to reconcile the common features of these two
theories inside a larger class rather than in the intersection of the two
classes. With a view to doing this, we defined a class of {\it rough CAT(0)
spaces} (abbreviated rCAT(0)) in \cite{BF1}, where we also investigated the
interior (i.e.~non-boundary) geometry of such spaces. This new class of
length spaces is arguably the smallest natural class of spaces that properly
contains all Gromov hyperbolic length spaces and all CAT(0) spaces; it is
not assumed that the spaces involved are geodesic, proper, or even complete.
Rough CAT(0) is closely related to the class of bolic spaces of Kasparov and
Skandalis \cite{KS1}, \cite{KS2} that was introduced in the context of their
work on the Baum-Connes and Novikov Conjectures, and is also related to
Gromov's class of CAT(-1,$\e$) spaces \cite{G}, \cite{DG}. They are closed
under reasonably general limit processes such as pointed and unpointed
Gromov-Hausdorff limits and ultralimits, and the rCAT(0) condition is
equivalent to a purely metric rough $n$-point condition for $n\ge 5$
\cite{BHn}.

Building on \cite{BF1}, we investigate the boundary theory of rCAT(0) spaces
in this paper. Unlike complete CAT(0) spaces, geodesic rays in an rCAT(0)
space do not form the basis of a nice boundary theory, and completeness is
not a useful assumption. Instead we replace geodesic rays by bouquets of
short paths whose lengths tend to infinity; one version of these bouquets is
closely related to the {\it roads} that V\"ais\"al\"a \cite{Va} introduced
in the context of Gromov hyperbolic spaces. We then define what we call the
{\it bouquet boundary} $\bbd X$ of $X$, and the associated bordification
$\bcl X:= X\cup \bbd X$. Moreover we define a {\it bouquet topology}
$\tau_B$ on $\bcl X$, denoting the corresponding subspace topology on $\bbd
X$ also by $\tau_B$. Similarly, we write $\icl X=X\cup\ibd X$ and $\Gcl
X=X\cup\Gbd X$ for the ideal and for the Gromov bordifications defined in
Section~\ref{S:prelims}.

The following pair of results show that the bouquet boundary with its
associated topology is indeed the desired type of generalization.

\begin{thm}\label{T:CAT0 main}
Suppose $X$ is a complete CAT(0) space. Then $\icl X$ equipped with the cone
topology and $\bcl X$ equipped with the bouquet topology are naturally
homeomorphic.
\end{thm}

\begin{thm}\label{T:Gromov main}
Suppose $X$ is a $\d$-hyperbolic length space, $\d\ge0$. Then $\bcl X$
equipped with the bouquet topology and $\Gcl X$ equipped with the canonical
topology are naturally homeomorphic.
\end{thm}

The rest of this paper is organized as follows. After some preliminaries in
Section~\ref{S:prelims}, Section~\ref{S:rcat0} reviews the parts of the
basic theory of rCAT(0) spaces developed in \cite{BF1} that are needed here.

In Section~\ref{S:bouquet}, we investigate several definitions of the
bouquet boundary, all defined using equivalence classes of bouquets of
paths, and prove their equivalence as sets, i.e.~there is a natural
bijection between any pair of them. We also relate the bouquet, ideal, and
end boundaries, and prove the following result.

\begin{thm}\label{T:end intro}
If $X$ is an unbounded proper rCAT(0) space, then $\bbd X$ is nonempty.
\end{thm}

Other possible definitions of the bouquet boundary use equivalence classes
of points ``tending to infinity'' (but not in the sense typically employed
for Gromov hyperbolic spaces). In Section~\ref{S:seq}, we prove that some
(but not all!) of these definitions are equivalent as sets to the
definitions in terms of path bouquets. We also show that in a Gromov
hyperbolic length space, all our sequential variants are equivalent as sets
to the Gromov boundary.

Finally in Section~\ref{S:top}, we define and investigate the bouquet
topology $\tau_B$, and prove the topological parts of the above results, as
well as the following result.

\begin{thm}\label{T:Hausdorff intro}
If $X$ is rCAT(0), then $\bcl X$ is Hausdorff and first countable. If
additionally $X$ is proper then both $\bcl X$ and $\bbd X$ are compact.
\end{thm}

\section{Preliminaries}\label{S:prelims}

Throughout this section, we suppose $(X,d)$ is a metric space. We say that
$X$ is {\it proper} if every closed ball in $X$ is compact.

We write $A\mino B$ and $A\maxo B$ for the minimum and maximum,
respectively, of two numbers $A,B$.

We define a {\it $h$-short segment} from $x$ to $y$, where $x,y\in X$, to be
a path of length at most $d(x,y)+h$, $h\ge 0$. A {\it geodesic segment} is a
$0$-short segment. $X$ is a {\it length space} if there is a $h$-short
segment between each pair $x,y\in X$ for every $h>0$, and $X$ is a {\it
geodesic space} if there is a geodesic segment between each pair $x,y\in X$.

A {\it geodesic ray} in $X$ is a path $\g:[0,\infty)\to X$ such that each
initial segment $\g|_{[0,t]}$ of $\g$ is a geodesic segment. The {\it ideal
boundary} $\ibd X$ of $X$ is the set of equivalence classes of geodesic rays
in $X$, where two geodesic rays $\g_1,\g_2$ are said to be equivalent if
$d(\tg_1(t),\tg_2(t))$ is uniformly bounded for all $t\ge 0$, where $\tg_i$
is the unit speed reparametrization of $\g_i$, $i=1,2$.

We refer the reader to \cite[Part II]{BHr} for the theory of CAT(0) spaces.
The ideal boundary $\ibd X$ of a complete CAT(0) space can be identified
with the set of geodesic rays from any fixed origin $o\in X$
\cite[II.8.2]{BHr}.

\begin{dfn}\label{D:cone top}
The {\it cone topology} $\ct$ on the {\it ideal bordification} $\icl X:=
X\cup \ibd X$ of a complete CAT(0) space $X$ is the topology with the
following neighborhood basis:
$$
\cB(x) =
  \begin{cases}
  \{B(x,r)  \mid r   > 0\}\,, & x\in X\,, \\
  \{U(x,r,t)\mid r,t > 0\}\,, & x\in \ibd X\,,
  \end{cases}
$$
where
\begin{align*}
B(x,r)   &= \{ y\in X\mid d(y,x)<r \} \\
U(x,r,t) &= \{ y\in \icl X\setminus B(o,r)\mid d(p_t(x),p_t(y))<r \}\,,
\end{align*}
and $p_t:\icl X\setminus B(o,r)\to X$ is the projection defined by the
intersection of the metric sphere $S(o,r)$ and the geodesic segment or ray
from $o$ to $x$.
\end{dfn}

We refer the reader to \cite{GH}, \cite{CDP}, \cite{Va}, or \cite[Part
III.H]{BHr} for the theory of Gromov hyperbolic spaces. We use the
non-geodesic definition: a metric space $(X,d)$ is {\it $\d$-hyperbolic},
$\d\ge 0$, if
$$ \ip xzw\ge \ip xyw\mino \ip yzw - \d\,, \qquad x,y,z,w\in X\,, $$
where $\ip xzw$ is the Gromov product\footnotemark{} defined by
\footnotetext{$\ip xyw$ is more commonly written as $\left<x,y\right>_w$.
Our notation is designed to avoid double subscripts.}%
$$ 2\ip xyw = d(x,w) + d(y,w) - d(x,y)\,. $$

Gromov sequences and the Gromov boundary have mainly been considered in
Gromov hyperbolic spaces, but they have also been defined in general metric
spaces \cite{BK}.

\begin{dfn}\label{D:Gromov seq}
A {\it Gromov sequence} in a metric space $X$ is a sequence $(x_n)$ in $X$
such that $\IP<x_m,x_n;o>\to\infty$ as $m,n\to\infty$. If $x=(x_n)$ and
$y=(y_n)$ are two such sequences, we write $(x,y)\in E$ if
$\IP<x_m,y_n;o>\to\infty$ as $m,n\to\infty$. Then $E$ is a reflexive
symmetric relation on the set of Gromov sequences in $X$, so its transitive
closure $\sim$ is an equivalence relation on the set of Gromov sequences in
$X$. Note that $E$ is already an equivalence relation if $X$ is Gromov
hyperbolic, but this is not true in general metric spaces \cite[1.5]{BK}.
The {\it Gromov boundary} $\Gbd X$ is the set of equivalence classes
$[(x_n)]$ of Gromov sequences.%
\end{dfn}

To simplify the statement of the following definition, we identify $x\in X$
with the singleton equivalence class $[(x_n)]$, where $x_n=x$ for all $n$.

\begin{dfn}\label{D:Gromov top}
The {\it Gromov bordification} $\Gcl X:= X\cup \Gbd X$ of a Gromov
hyperbolic space $X$ can be equipped with the {\it canonical topology} $\Gt$
that has the following neighborhood basis:
$$
\cG(x) =
  \begin{cases}
  \{B(x,r)\mid r > 0\}\,, & x\in X\,, \\
  \{V(x,r)\mid r > 0\}\,, & x\in \Gbd X\,,
  \end{cases}
$$
where $B(x,r)$ is as in Definition~\ref{D:cone top} and
\begin{align*}
V(x,r) = \{ y\in&\Gcl X \mid \exists \text{ \rm Gromov sequences }
  (a_n),(b_n) \,:\, \\
  &\quad[(a_n)]=x,\, [(b_n)]=y,\,
  \text{ and } \liminf_{m,n\to\infty} \IP<a_m,b_n;o> > r
  \}\,,
\end{align*}
\end{dfn}

The topology $\Gt$ is often given only for $\Gbd X$ where it is associated
with a canonical gauge of metrics, but we do not need these metrics. However
$\Gt$ defined on all of $\Gcl X$ can be found in the literature: for
instance, $\Gt$ is equivalent to the topology of \cite[III.H.3.5]{BHr} (for
proper geodesic hyperbolic spaces, to that of \cite[p.~6]{KB}, and to the
topology $\mathcal T_1^*$ in \cite[5.29]{Va} (but it is coarser than
$\mathcal T^*$ also defined there).


\section{Rough CAT(0) spaces}\label{S:rcat0}

In this section we review various notions of rough CAT(0) introduced in our
first paper \cite{BF1}, as well as some rCAT(0) results that we need here.
Except where otherwise referenced or proved, proofs of statements in this
section can be found in \cite{BF1}, where the reader can also find a more
detailed discussion of the concepts introduced below.

For the following definitions of short triangles and comparison points, we
denote $h$-short segments connecting points $x,y \in X$ by $[x,y]_h$. We use
the notation $[x,y]_h$ also for the image of this path, so instead of
$z=\g(t)$ for some $0\le t\le L$, we write $z\in[x,y]_h$. Given such a path
$\g$ and point $z=\g(t)$, we denote by $[x,z]_h$ and $[z,y]_h$ the subpaths
$\g|_{[0,t]}$ and $\g|_{[t,L]}$, respectively, both of which are also
$h$-short segments. This notation is ambiguous: given points $x,y$ in a
length space $X$ with at least two points, there are always many short
segments $[x,y]_h$ for each $h>0$. However the choice of $[x,y]_h$, once
made, does not affect the truth of the underlying statement.%

A \emph{$h$-short triangle} $T:=T_h(x_1,x_2,x_3)$ with vertices
$x_1,x_2,x_3\in X$ is defined as a collection of $h$-short segments
$[x_1,x_2]_h$, $[x_2,x_3]_h$ and $[x_3,x_1]_h$, and a \emph{comparison
triangle} is then a geodesic triangle $\bT:=T(\bx_1,\bx_2,\bx_3)$ in the
model space, Euclidean $\R^2$, so that $|\bx_i-\bx_j|=d(x_i,x_j)$,
$i,j\in\{1,2,3\}$. We say that $\bu\in \bT$ is a \emph{$h$-comparison point}
for $u\in T$, say $u\in [x_1,x_2]_h$, if
$$
|\bx-\bu| \leq \len([x,u]_h)
  \quad\text{and}\quad
|\bu-\by| \leq \len([u,y]_h)\,.
$$
Note that $\bu$ is not uniquely determined by $u$, but we do have
$$
|\bx-\bu| \geq \len([x,u]_h) - h
  \quad\text{and}\quad
|\bu-\by| \geq \len([u,y]_h) - h.
$$
Given a $h$-short triangle $T:=T_h(x,y,z)$ in any length space $X$, and
$u\in T$, we can always find a comparison triangle and $h$-comparison point
in $\R^2$.

Let $C\ge0$, and $h\ge0$. Suppose $T_h(x,y,z)$ is a $h$-short triangle in
$X$. We say that $T_h(x,y,z)$ satisfies the \emph{$C$-rough CAT(0)
condition} if given a comparison triangle $T(\bx,\by,\bz)$ in $\R^2$
associated with $T_h(x,y,z)$, we have
$$
d(u,v) \leq |\bu - \bv| + C\,,
$$
whenever $u,v$ lie on different sides of $T_h(x,y,z)$ and $\bu,\bv \in
T(\bx,\by,\bz)$ are corresponding $h$-comparison points.

\begin{dfn}
We say that the length space $(X,d)$ is \emph{$C$-\rcat(0)}, $C>0$, if
$T_h(x,y,z)$ satisfies the $C$-rough CAT(0) condition whenever $T_h(x,y,z)$
is a $h$-short triangle in $X$ with
\begin{equation}\label{E:H}
h\le H(x,y,z) = \frac{1}{1\maxo d(x,y)\maxo d(x,z)\maxo d(y,z)} \,.
\end{equation}
\end{dfn}

We omit the \emph{roughness constant} $C$ in the above notation if its value
is unimportant.

Our specific choice of $H$, although often useful, seems somewhat contrived.
A more natural definition would be to assume that there exists some
$H:X\times X\times X\to(0,\infty)$ such that the $C$-\rcat(0) condition
holds for $T_h(x,y,z)$ whenever $h\le H(x,y,z)$. We call this the {\it
$C$-\hrcat(0) condition}, $C>0$. It is formally weaker than the $C$-\rcat(0)
condition, but the two definitions are equivalent in the sense that
a $C$-\hrcat(0) space is $C'$-\rcat(0), with $C'=3C+2+\sqrt{3}$. %
Outside of esthetics, another advantage of the $C$-\hrcat(0) condition is
that it is an interesting condition for $C$ near $0$, unlike $C$-\rcat(0);
see Proposition~\ref{P:cat0 is rcat0}.

To ensure that CAT(0) spaces (or even just the Euclidean plane) are rCAT(0)
spaces, we need $h$ to be bounded by at most a fixed multiple of the above
function $H$; see \cite[Example 3.3]{BF1}. In particular, one cannot pick a
constant bound for $h$. Combining Theorem 4.5 and Corollary 4.6 of
\cite{BF1}, we do however get the following result.

\begin{prop}\label{P:cat0 is rcat0}
A CAT(0) space $X$ is $C$-\rcat(0) with $C=2+\sqrt{3}$, and $C$-\hrcat(0)
for all $C>0$.
\end{prop}

The analogous relationship with Gromov hyperbolic spaces is given by the
following result, which follows from the proof of \cite[Theorem~3.18]{BF1}.

\begin{prop}\label{P:Gh is rcat0}
A $\d$-hyperbolic length space, $\d\ge 0$, is $C$-\rcat(0) with $C=2+4\d$.
\end{prop}

The CAT(0) condition is equivalent to a weaker version of itself where the
comparison inequality is assumed only when one point is a vertex, and one
can even restrict the other point to being the midpoint of the opposite
side. Analogously \emph{weak} and \emph{very weak $C$-\rcat(0) spaces} are
defined by making the corresponding changes to the above definitions of
$C$-\rcat(0) spaces. Trivially an rCAT(0) space is weak rCAT(0), and a weak
rCAT(0) space is very weak rCAT(0), but we cannot at this time determine the
truth of the reverse implications.

We will not use the weak and very weak \rcat(0) variants in this paper, but
let us mention that the very weak variant is quantitatively equivalent to
the notion of bolicity introduced by Kasparov and Skandalis \cite{KS1},
\cite{KS2}; see \cite[Proposition 3.11]{BF1}.

The weak $C$-\rcat(0) condition can be written in the following more
explicit form: if $u=\la(s)$, where $\la:[0,L]\to X$ is a $h$-short path
from $y$ to $z$ parametrized by arclength, $h$ satisfies the usual bound,
and $0\le t\le 1$ is such that $td(y,z)\le s$ and $(1-t)d(y,z)\le L-s$, then
\begin{equation}\label{E:wrCAT0}
(d(x,u)-C)^2 \le (1-t)(d(x,y))^2+t(d(x,z))^2-t(1-t)(d(y,z))^2\,.
\end{equation}
This inequality holds {\it a fortiori} in $C$-rCAT(0) spaces, a fact that
will be useful later. Note that \eqref{E:wrCAT0} follows easily from the
definition of weak \rcat(0) and the following easily proved equality in the
Euclidean plane for a triangle with vertices $x,y,z$ and a point $u$ on the
side $yz$ such that $|y-u|=t|y-z|$:
\begin{equation*}
|x-u|^2 = (1-t)|x-y|^2+t|x-z|^2-t(1-t)|y-z|^2\,.
\end{equation*}

We have the following rough convexity lemma for rCAT(0) spaces.

\begin{lem}\label{L:rough cx}
Suppose $a_1,a_2,b_1,b_2$ are points in a $C$-rCAT(0) space $X$. Let
$\g_i:[0,1]\to X$ be constant speed $h_i$-short paths from $a_i$ to $b_i$,
$i=1,2$, where $h_i\le 1/(1\maxo d(a_i,b_i))$. Then there exists a constant
$C'$ such that
$$
d(\g_1(t),\g_2(t))\le (1-t)d(a_1,a_2)+td(b_1,b_2)+C'\,.
$$
In fact we can take $C'=2C$, and if either $a_1=a_2$ or $b_1=b_2$, we can
take $C'=C$. If $X$ is CAT(0), we can take $C'$ to be any positive number if
we add the restriction that $h_1,h_2\le\e$, where $\e=\e(C',d(a_1,b_1)\maxo
d(a_2,b_2))>0$ is sufficiently small.
\end{lem}

Except for its last statement, the above lemma is just a restatement of
\cite[Lemma~4.7]{BF1}. The last statement follows from the corresponding
convexity result for geodesic segments in a CAT(0) space (which states that
the estimate of Lemma~\ref{L:rough cx} holds with $C'=0$) and the fact that
a $h$-short path between any fixed pair of points $x,y$ in a CAT(0) space is
forced to stay arbitrarily close to the geodesic segment between these
points as long as both $h$ and $hd(x,y)$ are sufficiently small
\cite[Theorem~4.5]{BF1}.

Lastly we state and prove two lemmas that we will need in Sections
\ref{S:bouquet} and \ref{S:top}.

\begin{lem}\label{L:o1o2}
Suppose $o_1,o_2,u_1,u_2,x$ are points in a $C$-rCAT(0) space $X$ such that:
\begin{enumerate}
\item For $i=1,2$, $u_i$ lies on a path $\g_i$ of length $L_i$ from
    $o_i$ to $x$;
\item For $i=1,2$, $d(o_i,u_i)\le d(o_i,x)$;
\item $d(o_1,u_1)=d(o_2,u_2)$;%
\item For $i=1,2$, $\g_i$ is $h$-short, where $h:=H(o_1,o_2,x)$ and $H$
    is as defined in \eqref{E:H}.
\end{enumerate}
Then $d(u_1,u_2)\le C+d(o_1,o_2)$.
\end{lem}

\begin{proof}
Let $T:=T(o_1,o_2,x)$ be a $h$-short triangle such that $\g_1,\g_2$ are two
of its sides, and let $\bT:=T(\bo_1,\bo_2,\bx)$ be a comparison triangle.
Write $\d:=d(o_1,u_1)=d(o_2,u_2)$. For $i=1,2$, let $\bu_i$ be the point on
$[\bo_i,\bx]$ with $|\bo_i-\bu_i|=\d$. We claim that $|\bu_1-\bu_2|\le
|\bo_1-\bo_2|$. Since it is readily verified that $\bu_i$ is a
$h$-comparison point for $u_i$, the desired conclusion follows by applying
the $C$-rCAT(0) condition to this claimed inequality.

If the sidelengths $a,b,c$ of a Euclidean triangle $T(t)$ are changing with
time $t$ in such a way that $b'(t)=c'(t)=1$, and if the angle $A$ opposite
the side of length $a$ is constant, then differentiating the cosine rule
gives
$$ 2aa' = 2b + 2c - 2(b+c)\cos A $$
which immediately gives $a'(t)\ge 0$. Applying this fact with $b$ increasing
from $|\bu_1-\bx|$ to $|\bo_1-\bx|$, $c$ increasing from $|\bu_2-\bx|$ to
$|\bo_2-\bx|$, and $a$ changing from $|\bu_1-\bu_2|$ to $|\bo_1-\bo_2|$, the
claim follows.
\end{proof}

\begin{rem}\label{R:o1o2}
If we replace assumption (c) with the assumption that
$\len(\la_1)=\len(\la_2)$, where $\la_i$ is the subpath of $\g_i$ from $o_i$
to $u_i$, then we can take as comparison points the points $\bu_i$ on
$[o_i,u_i]$ for which $|\bo_i-\bu_i|=\len(\la_i)$, $i=1,2$. The conclusion
of Lemma~\ref{L:o1o2} now follows in the same manner.
\end{rem}

\begin{lem}\label{L:x1x2}
Suppose $o,u_1,u_2,x_1,x_2$ are points in a $C$-rCAT(0) space $X$ and that:
\begin{enumerate}
\item there exists $s\ge 0$ such that for $i=1,2$, $u_i=\g_i(s)$ for
    some unit speed path $\g_i$ of length $L_i$ from $o$ to $x_i$;
\item for $i=1,2$, $\g_i$ is $h$-short, where $h:=H(o,x_1,x_2)$ and $H$
    is as defined in \eqref{E:H}.
\end{enumerate}
Then $d(u_1,u_2)\le C+d(x_1,x_2)$.
\end{lem}

\begin{proof} As for Lemma~\ref{L:o1o2}, the proof reduces to an estimate for
planar triangles. Specifically, we claim that if $T=T(o,x_1,x_2)$ is a
triangle in the Euclidean plane and if $u_i\in[o,x_i]$ with $|o-u_i|=s$ for
$i=1,2$, then $|u_1-u_2|\le |x_1-x_2|$. By symmetry, it suffices to
establish this claim for the case $|o-x_1|\le|o-x_2|$. Since $|u_1-u_2|$ is
a linear function of $s$, we may as well assume that $u_1=x_1$. By
considering a triangle $T(t)$ as in the proof of Lemma~\ref{L:o1o2}, this
again follows by calculus.
\end{proof}


\section{Bouquet constructions}\label{S:bouquet}

In this section we first introduce the various concepts required to define
several variant {\it bouquet boundaries} of an rCAT(0) space $X$. We then
show that all of these notions can be identified in a natural way with each
other. Next we explore the relationship between the bouquet boundary and the
ideal boundary, showing that they can be naturally identified in a complete
CAT(0) space. Finally, we explore the relationship between ends and the
bouquet boundary, and prove Theorem~\ref{T:end intro}.

As motivation for the bouquet boundary, suppose $\g:[0,\infty)\to X$  is a
geodesic ray parametrized by arclength in an rCAT(0) space $X$, with
$\g(0)=o$. One of the basic properties of a complete CAT(0) space that we
would like to emulate is that if $o'\in X$ is any other point, then there is
a unit speed geodesic ray $\g':[0,\infty)\to X$ with $\g'(0)=o'$ and
$\sup_{t\ge 0} d(\g(t),\g'(t))<\infty$. The standard proof of this involves
taking a sequence of geodesic segments from $o'$ to $\g(t_n)$ where
$t_n\uparrow\infty$. The resulting unit speed paths $\g_n:[0,L_n]\to X$ are
such that $d(\g_m(t),\g_n(t))$ is uniformly bounded for all $m,n\in\N$ and
all $0\le t\le L_m\mino L_n$. Moreover if we fix $t$ and pick $m,n\ge n_0$,
then this uniform bound on $d(\g_m(t),\g_n(t))$ tends to $0$, and $L_m\mino
L_n\to\infty$ as $n_0\to\infty$. Defining $\g'(t)=\lim_{n\to\infty}\g_n(t)$
for all $t\ge0$ gives a geodesic ray $\g'$ from $o'$.

If $X$ is merely rCAT(0) and if we use $h_n$-short paths $\g_n$ for some
appropriately small positive numbers $h_n$, then we can similarly derive a
uniform bound on $d(\g_m(t),\g_n(t))$. However the rCAT(0) condition does
not imply that this bound tends to $0$ for $m,n\ge n_0\to\infty$, so
completeness is of no use. To overcome this obstacle, we discard geodesic
rays and instead construct a boundary using sequences of paths such as
$(\g_n)$ above. The key features of $(\g_n)$ are that all segments $\g_n$
have a common origin, their lengths are increasing and tending to infinity
(this may require that we take a subsequence above), and
$d(\g_m(t),\g_n(t))$ is uniformly bounded whenever it is defined.

Bouquets $(\g_n)$ with a uniform bound on $d(\g_m(t),\g_n(t))$ are the most
natural concept arising from the above considerations, and are closely
related to the {\it roads} that V\"ais\"al\"a \cite{Va} introduced in the
context of Gromov hyperbolic spaces. There are two useful variants of this
concept that lead to a naturally equivalent bouquet boundary. The first is a
loose bouquet, where the bound on $d(\g_m(t),\g_n(t))$ is not uniform, but
is instead allowed to grow more slowly than the smaller of the two distances
$d(o',\g_m(t))$ and $d(o',\g_n(t))$; such loose bouquets are needed in the
next section to investigate sequential versions of the bouquet boundary. The
second equivalent notion is a standard bouquet, a tighter notion than a
bouquet which is needed to define the bouquet topology in
Section~\ref{S:top}.

\subsection{Bouquets: Definitions and basics}

\begin{dfn}\label{D:little o}
A {\it little-o function} is a monotonically increasing function
$\d:[0,\infty)\to[0,\infty)$ such that $\d(t)/t\to 0$ as $t\to\infty$, and
$|\d(s)-\d(t)|\le |s-t|$ for all $s,t>0$.
\end{dfn}

\begin{dfn}\label{D:short}
A {\it short function} is a decreasing function $D:[0,\infty)\to(0,1]$,
satisfying $D(t)\le 1/t$ for $t>1$, and $|D(s)-D(t)|\le|s-t|$ for all
$s,t>0$.
\end{dfn}

The $1$-Lipschitz condition forms part of both above definitions for
technical reasons: in the case of Definition~\ref{D:little o}, it is used in
the next section to force $\d(\len(\g))$ to be close to $\d(d(o,x))$ when
$\g$ is a $1$-short path from $o$ to $x$, while in the case of
Definition~\ref{D:short}, it ensures that subsegments of a $D$-short segment
are $D$-short (see Definition~\ref{D:D short} and Lemma~\ref{L:subseg}).

Note that the Lipschitz assumption in Definition~\ref{D:short} does not
restrict the decay rate of a short function: if $E:[0,\infty)\to(0,1]$ is
any decreasing function such that $E(t)\le 1/t$ for $t>1$, and
$D:[0,\infty)\to (0,1]$ is the function which is affine on each interval
$[n-1,n]$, $n\in\N$, and defined by the equation $D(n-1)=E(n)$, then $D$ is
a short function. Similarly the Lipschitz assumption in
Definition~\ref{D:little o} puts no restriction on how slowly or quickly
$\d$ increases, among the class of monotonically increasing functions
satisfying $\d(t)/t\to 0$, since we could define such a $\d$ by piecewise
linear interpolation of the values of a function $f$ at $0$ and $2^{n-1}A$,
$n\in\N$, where $f$ is any given non-negative function satisfying $f(t)/t\to
0$ as $t\to\infty$ and $A>0$ is so large that $f(s)/s\le 1/2$ for $t\ge A$.

\begin{dfn}\label{D:D short}
Given a short function $D$, a segment from $x$ to $y$ is said to be {\it
$D$-short} if it is $h$-short for $h=D(d(x,y))$.%
\end{dfn}

Note that, although the two concepts of $h$-short and $D$-short segments
create a potential ambiguity of terminology, the context will always
indicate which sense of ``short`'' we mean, and we also use the convention
of using capital or lower-case letters to indicate whether we are talking
about a short segment in this new sense or the old sense, respectively.

\begin{lem}\label{L:subseg}
Every subsegment of a $D$-short segment is a $D$-short segment.
\end{lem}

\begin{proof}
Suppose $\g:[0,L]\to X$ is a $D$-short segment from $x$ to $y$ parametrized
by arclength, and let $h:=D(d(x,y))$, so that $L\le d(x,y)+h$. Let
$z_i:=\g(t_i)$ for some $t_i\in[0,L]$, $i=1,2$, let $\la$ be the associated
subpath of $\g$, and let $M=|t_1-t_2|$ be the length of $\la$. A subpath of
a $h$-short segment is a $h$-short segment (just use the triangle
inequality!), and so $\la$ is a $D$-short segment if $d(z_1,z_2)\le d(x,y)$.

If instead $\d:=d(z_1,z_2)-d(x,y)>0$, then
$$ M-d(z_1,z_2) \le L-d(z_1,z_2) = L-d(x,y)-\d \le h-\d $$
while $D(d(z_2,z_2))\ge h-\d$ by the Lipschitz property, and so $\la$ is
again $D$-short.
\end{proof}

We are now ready to define our three variants of bouquets.

\begin{dfn}\label{D:loose bouquet}
Suppose $X$ is an rCAT(0) space. Let $\d$ be a little-o function, $D$
a short function, and $o\in X$. A \emph{loose $(\d,D)$-bouquet from $o$} is
a sequence $\b$ of unit speed $D$-short segments $\b_n:[0,L_n]\to X$, $n \in
\N$, with the following properties:
\begin{itemize}
\item[(i)] $\b_n(0)=o$ for all $n\in\N$; we call $o$ the \emph{initial
    point} of $\b$.
\item[(ii)] $(L_n)$ is monotonically increasing and has limit infinity.
\item[(iii)] $d(\b_m(t),\b_n(t))\le \d(t)$, for all $0\le t\le L_m$,
    $m\le n$, $m,n\in\N$.
\end{itemize}
We call the points $\b_n(L_n)$ the {\it tips of $\b$}.
\end{dfn}

\begin{dfn}\label{D:bouquet}
A {\it $(c,D)$-bouquet from $o$} is a loose $(\d,D)$-bouquet from $o$ for
some constant function $\d(t)\equiv c\ge 0$. A {\it standard bouquet from
$o$} in a $C$-rCAT(0) space means a $(2C+2,D)$-bouquet $\b$ from $o$, with
$D(t)=1/(1\maxo(2t))$, $t\ge 0$, and $L_n=(2C+2)^n$, $n\in\N$.
\end{dfn}

Note that the definition of a standard bouquet depends on the $C$-parameter
of the ambient rCAT(0) space. In this definition, the precise choice
$c=2C+2$ is a mere convenience, but choosing some fixed $c>2C$ is important
for the topological arguments in Section~\ref{S:top}. As for $L_n$, it is
only important that we choose some sequence increasing to infinity but
$L_n=(2C+2)^n$ is technically convenient.

We often speak of {\it (loose/standard) bouquets}, dropping references to
the initial point $o$ and parameters $c,\d,H$, if these are unimportant. We
denote by $\nbq$, $\lbq$, $\sbq$, the sets of all bouquets, loose bouquets,
or standard bouquets (with basepoint $o$), respectively, so
$\sbq\subset\nbq\subset\lbq$.

\begin{dfn}\label{D:loose-bouquet-bdy}
Let $\b^i=(\b_n^i)_{n=1}^\infty$, $i=1,2$, be a pair of loose bouquets in an
rCAT(0) space $X$, where $\b_n^i:[0,L_n^i]\to X$. Then $\b^1$ and $\b^2$ are
said to be \emph{loosely asymptotic}, denoted $\b^1 \sim_\rL \b^2$, if there
is a little-o function $\d$ such that
$$
d(\b_m^1(t),\b_n^2(t))\le \d(t)\,, \qquad
  0\le t\le L_m^1\mino L_n^2\,.
$$
The equivalence class of loose bouquets loosely asymptotic to $\b$ will be
denoted by $[\b]_{\rL}$.
\end{dfn}

\begin{dfn}\label{D:bouquet-bdy}
Bouquets $\b^1$ and $\b^2$ are said to be \emph{asymptotic}, denoted $\b^1
\sim \b^2$, if they are loosely asymptotic for some constant little-o
function $\d(t)\equiv K\ge 0$. The equivalence class of bouquets asymptotic
to $\b$ will be denoted by $[\b]$.
\end{dfn}

\begin{dfn}\label{D:bouquet-bdy-notation}
Assuming $X$ is an rCAT(0) space, we call $\lbbd X:=\lbq/\sym_\rL$ the {\it
loose bouquet boundary of $X$}, and $\bbd X:=\nbq/\sym$ the {\it bouquet
boundary of $X$}.
\end{dfn}

Other variants of interest are $\nbq/\sym_\rL$, $\sbq/\sym_\rL$, and
$\sbq/\sym$. We will see that all five variants lead to naturally equivalent
notions of a boundary at infinity (Corollary~\ref{C:5 equiv}), that they are
independent of the choice of basepoint $o$ (Corollaries \ref{C:std asymp}
and \ref{C:5 equiv}), and that they generalize the ideal boundary of a
complete CAT(0) space (Theorem~\ref{T:CAT0 bij}).

It is clear that $\sim$ or $\sim_\rL$ is an equivalence relation in each of
the above five variants. Note that $\sim$ is not an equivalence relation on
$\lbq$ since the notion of asymptoticity must be at least as loose as the
bound on $d(\b_m(t),\b_n(t))$ in order to have an equivalence relation; easy
examples can be found in the Euclidean plane.

Let us pause to make a few remarks relating to the above definitions. First,
note that if $\b^1,\b^2$ are loose $(\d_1,D)$-bouquets and if
$d(\b_m^1(t),\b_n^2(t))\le \d(t)$ for one particular choice of $m,n$, then
it follows that $d(\b_m^1(t),\b_n^2(t))\le \d(t)+2\d_1(t)$ for all allowable
choices of $m,n$. So if we do not care about the particular little-o
function $\d$, we can write the loose asymptoticity condition as
$$ d(\b^1(t),\b^2(t))\le \d(t)\,, \qquad 0\le t<\infty\,, $$
where $\b^i(t)$ can be interpreted as $\b_n^i(t)$ for any single $n=n(t)$
for which $\b_n^i(t)$ is defined. Using Lemma~\ref{L:rough cx}, this last
inequality for fixed $t$ implies that
$$ d(\b^1(s),\b^2(s))\le \d(t)+2C\,, \qquad 0\le s\le t\,, $$
and so loose asymptoticity of $\b^1$ and $\b^2$ is equivalent to the
formally weaker condition: there exists a little-o function $\d$ such that
$$ \liminf_{t\to\infty}\frac{d(\b^1(t),\b^2(t))}{\d(t)}\le 1\,. $$

It follows routinely from the triangle inequality and the fact that we are
using $1$-short segments that the bound $d(\b_m(t),\b_n(t))\le \c$ in the
definition of a bouquet is quantitatively equivalent to assuming the
seemingly weaker condition $d_H(\b_m,\b_n|_{[0,L_m]})\le c'$, where $d_H$
indicates Hausdorff distance. In fact the latter condition for a given $c'$
implies the former condition for $c=2c'+1$. Similarly the uniform bound on
$d(\b_m^1(t),\b_n^2(t))$ in the definition of asymptotic bouquets is
quantitatively equivalent to a uniform bound on the Hausdorff distance
between $\b_m^1$ and $\b_n^2|_{[0,L_m^1]}$, assuming without loss of
generality that $L_m^1\le L_n^2$. Similar comments apply to the definitions
of loose bouquets and loose asymptoticity.

\begin{dfn}
Suppose $\a:=(\a_n)_{n=1}^\infty$ is a sequence of numbers, with $0<\a_n\le
1$, $n\in\N$, and suppose $\b=(\b_n)$ is a (loose) bouquet. The {\it
$\a$-pruning of $\b$} is $\b'=(\b_n')$, where $\b_n'=\b_n|_{[0,\a_nL_n]}$.
If $\a$ is a constant sequence $(a)$, we may refer to the {\it $a$-pruning
of $\b$} in place of the $\a$-pruning of $\b$.
\end{dfn}

We now make three simple observations about ways to get (loose) bouquets
(loosely) asymptotic to a given bouquet; we write ``equivalent'' in all
cases instead of ``(loosely) asymptotic''. The last of these three
observations is the only one where we needed to use the rCAT(0) condition,
specifically in the form of Lemma~\ref{L:rough cx}.

\begin{obs}\label{O:b subs}
Every subsequence of a (loose) bouquet $\b$ is a (loose) bouquet equivalent
to $\b$; we call such a subsequence a {\it (loose) sub-bouquet}.
\end{obs}

\begin{obs}
If $\b$ is a (loose) bouquet, then an $\a$-pruning of $\b$ is also a (loose)
bouquet as long as the sequence $(\a_nL_n)$ is increasing and has limit
infinity. Whenever the $\a$-pruning of $\b$ is a (loose) bouquet, it is
equivalent to $\b$. In particular if $0<a<1$, then the $a$-pruning of a
(loose) bouquet $\b$ is always a (loose) bouquet equivalent to $\b$.
\end{obs}

\begin{obs} If a (loose) bouquet $\b'$ from $o'$ has the same sequence of
tips as a (loose) bouquet $\b$ from $o$, then $\b$ and $\b'$ are (loosely)
asymptotic.
\end{obs}

\subsection{Equivalence of bouquet boundary definitions}

\begin{thm}\label{T:bouquet-oo'}
Let $X$ be a $C$-rCAT(0) space and let $\b$ be a $(c,D)$-bouquet from $o$ in
$X$. If $o'\in X$, $c'>C$, and $D'$ is any short function, then there exists
a $(c',D')$-bouquet from $o'$ which is asymptotic to $\b$.
\end{thm}

\begin{proof}
Let $\b=(\b_n)_{n=1}^\infty$ with $\b_n:[0,L_n]\to X$ as usual, and let
$x_n:=\b_n(L_n)$. 
Let $y_n:=\b_n(M_n)$ where $0\le M_n\le L_n$ is chosen so that
$d(o,y_n)=d(o,x_n)/2$. By thinning out $\b$ if necessary, we assume that
$L_1\ge 1+4d(o,o')$, and that
$$ L_{n+1}\ge L_n+4d(o,o')+3\,, \qquad n\in\N\,. $$
It follows that
$$ d(o,x_{n+1})\ge d(o,x_n)+4d(o,o')+2\,, \qquad n\in\N\,. $$
and so
$$
\left.
\begin{aligned}
d(o,y_{n+1})  &\ge d(o,y_n)+2d(o,o')+1 \\
d(o',y_{n+1}) &\ge d(o',y_n)+1
\end{aligned}
\right\}
\,, \qquad n\in\N\,.
$$
Also $d(o,x_1)\ge 4d(o,o')$, so $d(o,y_1)\ge 2d(o,o')$, and $d(o',y_1)\ge
d(o,o')$.

We choose a collection of unit speed $h_n'$-short paths $\la_n:[0,M_n']\to
X$ from $o'$ to $y_n$, where $h_n':=D'(d(o',y_n))/2$. Because
$$ d(o',y_n)\ge d(o',y_1)\ge d(o,o')\,, $$
we see that
$$ d(o,y_n)\le d(o,o')+d(o',y_n)\le 2d(o',y_n)\,, $$
and so
$$ h_n'\le 1/2d(o',y_n)\le 1/d(o,y_n)\,. $$
Since also $h_n'\le 1/d(o,o')$, we see that $h_n'\le H(o,o',y_n)$, where $H$
is as in \eqref{E:H}. The shortness parameter for $\b_n$ is
$$ D(d(o,x_n))\le 1/d(o,x_n)\le 1/2d(o,y_n)\,, $$
and we similarly see that $D(d(o,x_n))\le H(o,o',y_n)$. Because
$$ d(o',y_{n+1})\ge d(o',y_n)+1\,, $$
we see that the sequence $(M_n')$ is monotonically increasing. Also
$M_n'\to\infty$ simply because $L_n\to\infty$.

We now fix $n,m\in\N$ with $m\le n$, and choose $y_m^1$ on
$\b_n|_{[0,M_n]}$, and $y_m^2$ on $\la_n|_{[0,M_n']}$ so that
$d(y_m^1,o)=d(y_m^2,o)=d(y_m,o)$. By Lemma~\ref{L:o1o2}, we see that
$$ d(y_m^1,y_m^2)\le C+d(o,o')\,, $$
and the $(c,D)$-bouquet condition ensures that $d(y_m^1,y_m)\le 1+c$. Thus
\begin{equation}\label{E:y y2}
d(y_m,y_m^2)\le 1+c+C+d(o,o')\,.
\end{equation}
Using the $C$-rCAT(0) condition we
readily deduce that $\la:=(\la_n)$ is a $(c'',D')$-bouquet from $o'$ that is
asymptotic to $\b$, where $c''=1+c+2C+d(o,o')$. If $c''$ is larger than
$c'$, we simply replace $\la$ by the $a$-pruning of $\la$, where
$a=(c'-C)/(1+c+C+d(o,o'))$.
\end{proof}

\begin{rem}\label{R:bouquet-oo'}
The above proof works just as well if $(y_n)$ is any other sequence of
points such that $y_n$ lies on $\b_n$, $d(o,y_n)\le d(o,x_n)/2$, and
$(d(o,y_n))$ is unbounded, although we might need to select a subsequence to
ensure that $d(o,y_{n+1})$ and $d(o',y_{n+1})$ increase quickly enough.
Alternatively if $D(t)\le 1/(1\maxo 2t)$, we could use paths to $x_n$
instead of paths to $y_n$, since paths to $y_n$ were needed only to ensure
that the rCAT(0) condition could be applied to the resulting triangle. The
latter variant will prove useful in Section~\ref{S:top}.
\end{rem}

\begin{cor}\label{C:std asymp}
If $X$ is an rCAT(0) space, then $\sbq/\sym$ can be identified with $\bbd
X=\nbq/\sym$, and both are independent of the basepoint $o$.
\end{cor}

\begin{proof}
Let $\b$ be a bouquet from $o$ in $X$. By Theorem~\ref{T:bouquet-oo'}, there
exists a $(2C+2,D)$-bouquet $\b'$ from any other point $o'\in X$ that is
asymptotic to $\b$, where $D(t)=1/(1\maxo(2t))$, $t\ge 0$. If the associated
lengths $L_n'$ of $\b_n'$ are not as required, taking a subsequence allows
us to assume that they are at least as large as required, and then we get a
standard bouquet by suitably pruning this bouquet. The result now follows
easily.
\end{proof}

We next prove the equivalence of the bouquet boundary and the loose bouquet
boundary.

\begin{thm}\label{T:bouq bij 1}
If $X$ is an rCAT(0) space, then the identity map from $\nbq$ to $\lbq$
induces a natural bijection $\il:\bbd X\to\lbbd X$.
\end{thm}

\begin{proof}
Suppose $X$ is $C$-rCAT(0). Trivially $\il([\b]):=[\b]_{\rL}$ is
well-defined. We next prove that $\il$ is injective. Suppose that
$\il([\b^1])=\il([\b^2])$ for a pair of $(c,D)$-bouquets $\b^1,\b^2$. By
Theorem~\ref{T:bouquet-oo'}, we may assume that $\b^1,\b^2$ have a common
initial point $o$. Then there exists a little-o function $\d$ such that
$$
d(\b_{m_t}^1(t),\b_{n_t}^2(t))\le \d(t)\,,
$$
where the indices $m_t,n_t$ are such that $L_{m_t}^1\mino L_{n_t}^2\ge t$,
where $L_m^i=\len(\b_m^i)$ as usual. Using Lemma~\ref{L:rough cx}, we deduce
that for all indices $m,n\in\N$,
$$
d(\b_m^1(s),\b_n^2(s))\le s\frac{\d(t)}{t}+2C+2c\,,\qquad
  0\le s\le L_m^1\mino L_n^2\mino t\,.
$$
Letting $t$ tend to infinity, we deduce that
$$
d(\b_m^1(s),\b_n^2(s))\le 2C+2c\,,\qquad
  0\le s\le L_m^1\mino L_n^2\,,
$$
and so $[\b^1]=[\b^2]$, as required.

Finally, we prove that $\il$ is surjective. Suppose $\b=(\b_n)$ is a loose
$(\d,D)$-bouquet, with $\b_n:[0,L_n]\to X$, $n\in\N$ as usual. We choose a
strictly increasing sequence of positive integers $(n_k)_{k=1}^\infty$ such
that $L_{n_k}\ge k$ and $k\d(L_{n_k})\le L_{n_k}$. We let
$\b'=(\b_k')_{k=1}^\infty$ be the $\a$-pruning of the sub-bouquet
$(\b_{n_k})$, where $\a=(\a_k)_{k=1}^\infty$ is defined by $\a_k=k/L_{n_k}$.
Then $\b_k'$ has length $k$. Using Lemma~\ref{L:rough cx}, we deduce that
for all $m\le n$,
$$
d(\b_m'(s),\b_n'(s))\le C+\a_m\d(L_{n_m})\le C+1\,,\qquad
  0\le s\le m\,,
$$
Thus $\b'$ is a $(1+C,D)$-bouquet and, since it is a pruning of a
sub-bouquet of $\b$, it is loosely asymptotic to $\b$. Thus
$\il([\b'])=[\b]_{\rL}$, as required.
\end{proof}

\begin{cor}\label{C:5 equiv}
There are natural identification maps between the boundary variants
$\lbq/\sym_\rL$, $\nbq/\sym_\rL$, $\sbq/\sym_\rL$, $\nbq/\sym$, and
$\sbq/\sym$ of an rCAT(0) space $X$.
\end{cor}

\begin{proof} Let $q_1:\nbq\to\bbd X = \nbq/\sym$ and $q_2:\lbq\to\lbbd X =
\lbq/\sym_\rL$ be the defining quotient maps. By Corollary~\ref{C:std
asymp}, $q_3:=q_1|_{\sbq}:\sbq\to\bbd X$ is surjective. By
Theorem~\ref{T:bouq bij 1}, the identity map $i:\nbq\hookrightarrow\lbq$
induces a natural identification $\il:\bbd X\to\lbbd X$, and so the
following diagram commutes.
$$
\xymatrix{
\sbq \ar@{->>}[rd]_{q_3} \ar@{^{(}->}[r] &
 \nbq \ar@{->>}[d]^{q_1} \ar@{^{(}->}[r]^{i} &
 \lbq \ar@{->>}[d]^{q_2} \\
& \bbd X \ar@{^{(}->>}[r]^{\il} & \lbbd X
}
$$
Each of the five types of boundary is either the image of $q_i$, $1\le i\le
3$, or the image of $i_L\circ q_i$, $i=1,3$, and they can all be identified
because the maps $q_i$ are all surjective and $\il$ is bijective.
\end{proof}

\subsection{The ideal boundary versus the bouquet boundary}

In an rCAT(0) space $X$, a geodesic ray $\g:[0,\infty)\to X$ can be
identified with the bouquet $(\g_n)_{n=1}^\infty$, where $\g_n$ is the
initial segment of length $n$ of $\g$, parametrized by arclength. This gives
rise to a natural injection $\ii:\ibd X\to\bbd X$.

\begin{thm}\label{T:CAT0 bij}
If $X$ is a complete CAT(0) space, then the natural injection $\ii:\ibd
X\to\bbd X$ is bijective.
\end{thm}

Before proving Theorem~\ref{T:CAT0 bij}, we give simple examples to show
that the ideal boundary is not as well behaved in rCAT(0) spaces, or even in
incomplete CAT(0) spaces, as it is in complete CAT(0) spaces. Such
pathologies would be known to experts. The point of giving them here is to
contrast them with Theorem~\ref{T:bouquet-oo'} which says that no such
pathologies arise with the bouquet boundary of an rCAT(0) space.

\begin{exa}\label{X:CAT0 1}
Let $X$ be the metric subspace of the Euclidean plane given as follows using
Cartesian coordinates:
$$ X=\{(0,0)\}\cup (0,1)\times(0,\infty)\,. $$
Then, as a convex subset of the Euclidean place, $X$ is CAT(0). It is also
clear that $\bbd X$ is a singleton set, as is $\ibd X$ if we define it as
the set of equivalence classes of geodesic rays. There is however no
geodesic ray from $o=(0,0)$.
\end{exa}

\begin{exa}\label{X:CAT0 2}
Let $X_i$ be an isometric copy of the space $X$ in Example~\ref{X:CAT0 1},
and let $Y_I$ be the metric space obtained as a quotient of the disjoint
union $\bigcup_{i\in I} X_i$ where we identify every copy of $(0,0)$, and
$I$ is some nonempty index set. Then $Y_I$ is CAT(0) and there is a natural
bijection from $\bbd {Y_I}$ to $I$; the same can be said of $\ibd X$ if we
define it as the set of equivalence classes of geodesic rays. However, there
is only one geodesic ray from each $y\in Y_I$, $y\ne o$, and none at all
from $o$.
\end{exa}

In the previous pair of examples, we could identify the ideal and bouquet
boundaries, even if the ideal boundary was not as well behaved. The next
example shows that the situation can be worse than this.

\begin{exa}\label{X:CAT0 3}
Let $X$ be the subset of the Euclidean plane given as follows
$$
X = \(\bigcup_{i=0}^\infty\{(i,0)\}\) \cup
    \(\bigcup_{i=1}^\infty (i-1,i)\times (0,1)\)\,.
$$
Then we claim that $X$ is CAT(0). To see this, suppose that $x,y,z$ are
fixed but arbitrary points in $X$. If these points can be connected by a
geodesic triangle, then any such triangle must clearly be contained in some
``initial part'' of $X$ having the form $X_n=\bigcup_{i=1}^n A_i$, where
$$
A_i = \{(i-1,0),(i,0)\} \cup (i-1,i)\times (0,1)\,, \qquad 1\le i\le n\,.
$$
Since each $A_i$ is a convex subset of the plane, it is CAT(0) in the
induced metric. Since $X_n$ is obtained by a finite succession of isometric
gluings of the sets $A_i$ along closed convex subsets (in fact along
singleton sets!), it follows from the basic gluing theorem II.11.1 of
\cite{BHr} that $X_n$ is CAT(0). In particular there exists at least one
geodesic triangle with vertices $x,y,z$, and all such triangles satisfy the
CAT(0) inequality. Since $x,y,z$ are arbitrary, we deduce that $X$ is
CAT(0). It is also clear that the bouquet boundary is a singleton set, but
that $X$ contains no geodesic ray. By joining isometric copies of $X$ at
$o=(0,0)$, we can get a space whose bouquet boundary has any desired
cardinality, but whose ideal boundary is empty.
\end{exa}

Theorem~\ref{T:CAT0 bij} is an immediate consequence of the following
generalization to bouquets of a well-known result concerning geodesic rays;
the proof is a modification of that of Theorem~\ref{T:bouquet-oo'}.

\begin{thm}\label{T:geod ray}
Let $X$ be a complete CAT(0) space and let $\b$ be a $(c,D)$-bouquet from
$o$ in $X$. Given any $o'\in X$, there exists a geodesic ray parametrized by
arclength $\g:[0,\infty)\to X$ with $\g(0)=o'$ and which is asymptotic to
$\b$ in the sense that there exists a constant $c'$ such that
$$
d(\g(t),\b_n(t))\le c'\,, \qquad 0\le t\le L_n\,, \; n\in\N\,,
$$
where as usual $L_n=\len(\b_n)$.%
\end{thm}

\begin{proof}
By Theorem~\ref{T:bouquet-oo'}, there exists a $(c_2,D')$-bouquet
$\b'=(\b_n')$ from $o'$ that is asymptotic to $\b$, where $c_2$ is as in the
proof of that result and $D'$ is an arbitrary short function to be fixed
below. If necessary, we take a subsequence of $\b'$ to ensure that the
associated sequence of path lengths $(L_n')$ are such that $L_{n+1}'\ge 4^n
L_n'>0$. For each $n\in\N$, we then prune $\b_n'$ by a factor $\a_n:=2^{-n}$
to get a path $\b_n''$ of length $L_n''$ which increases to infinity. For
$n_0\in\N$, the sequence $(\b_{n+n_0}'')_{n=0}^\infty$ is a
$(2^{-n_0+1}c_2,D')$-bouquet: this follows from Lemma~\ref{L:rough cx} with
$a_1=a_2=o'$. Note that the parameter $2^{-n_0+1}c_2$ is twice as large as
would be needed for $C'=0$ in order to incorporate the $C'$ term; for this
to suffice, we of course need that $C'>0$ be sufficiently small, but this
can be guaranteed by choosing a sufficiently small short function $D'$
above.

It follows that the sequence $\b''=(\b_n'')$ converges in the pointed
Hausdorff sense to a path $\g$. Since $\b_n''$ is $h_n$-short where
$h_n=D'(d(o,y_n))\to 0$ as $n\to\infty$, it follows that $\g$ is a geodesic
ray.

Since $\b'$, and so also $\b''$, is asymptotic to $\b$, it is clear that
$\g$ is asymptotic to $\b$.
\end{proof}

\subsection{The end boundary and the bouquet boundary}

By an {\it end} of a metric space $X$ (with basepoint $o$), we mean a
sequence $(U_n)$ of components of $X\setminus \bB_n$, where
$\bB_n=\overline{B(o,n)}$ for fixed $o\in X$ and $U_{n+1}\subset U_n$ for
all $n\in\N$. We do not require $\bB_n$ to be compact. We denote by $\ebd X$
the collection of ends of $X$ and call it the {\it end boundary of $X$}.

Ends with respect to different basepoints are compatible under set
inclusion: defining $U_n,V_n$ for all $n\in\N$ to be components of
$X\setminus\overline{B(o,n)}$ and $X\setminus\overline{B(o',n)}$,
respectively, it is clear that $U_n$ is a subset of a unique $V_m$ whenever
$n-m>d(o,o')$. This compatibility gives rise to a natural bijection between
ends with respect to different basepoints, allowing us to identify them and
treat the end boundary as being independent of the basepoint.

A {\it finite $\e$-net} for a subset $A$ of a metric space $X$ is a set
$S\subset X$ of finite cardinality such that every $x\in A$ lies in the ball
$B(x,\e)$ for some $x\in S$. Requiring $A$ to have a finite $\e$-net for
fixed $\e>0$ is strictly weaker than requiring $A$ to be totally bounded.

We now examine the relationship between the end boundary and the bouquet
boundary of an rCAT(0) space.

\begin{thm}\label{T:end}
If $X$ is an rCAT(0) space, then there is a natural map $\me:\bbd X\to\ebd
X$. If additionally there exists $\e>0$ such that every ball in $X$ has a
finite $\e$-net, then
\begin{enumerate}
\item $\me$ is surjective;
\item $\bbd X$ and $\ebd X$ are nonempty if and only if $X$ is
    unbounded.
\end{enumerate}
\end{thm}

The assumption that balls have finite $\e$-nets is essentially stating that
balls are ``totally bounded at a fixed scale $\e$''. This holds in
particular if $X$ is proper, so the above result implies Theorem~\ref{T:end
intro}. Without this assumption both conclusions (a) and (b) in
Theorem~\ref{T:end} can fail. In the case of (b) this is easy to see: the
union of segments from the origin to $(n,1)$, $n\in\N$, in the Euclidean
plane is unbounded but clearly both $\bbd X$ and $\ebd X$ are empty. For the
failure of (a), even for complete CAT(-1) spaces (a condition which implies
both CAT(0) and Gromov hyperbolic), we refer the reader to
\cite[Theorem~2]{BF2}.

\begin{proof}[Proof of Theorem~\ref{T:end}]
Let $\cU_n$ denote the set of components of $X\setminus\overline{B(o,n)}$
for fixed $o\in X$, and let $\b=(\b_n)$ be a bouquet with initial point $o$,
where $\b_n:[0,L_n]\to X$ as usual.  We claim that for each $m\in\N$ there
exists $t_0=t_0(m)>0$ and $U_m\in\cU_m$ such that $\b_n(t)\in U_m$ whenever
$n\in\N$ and $t_0<t\le L_n$. In fact if $\b$ is a $(c,D)$-bouquet from $o$,
then we can take $t_0=m+3/2+c/2$.

To justify the claim, suppose first that $x:=\b_n(t)\in U_m$ and
$x':=\b_n(t')\in U_m'$, where $U_m,U_m'$ are distinct elements of $\cU_m$,
and $t'>t>t_0=m+3/2+c/2$. Any path from $x$ to $x'$ must pass through
$\overline{B(o,m)}$, and $\b_n$ is $1$-short, so
$$ d(x,x') \ge t-1-m+t'-1-m = t'-t+2(t-m-1)> t'-t+1\,, $$
contradicting the fact that the segment $\b|_{[t,t']}$ is $1$-short.

Suppose next that $x:=\b_n(t)\in U_m$ and $x':=\b_{n'}(t)\in U_m'$, where
$U_m,U_m'$ are distinct elements of $\cU_m$, and $t'>t>t_0=m+3/2+c/2$. As
before, $d(x,x')\ge 2(t-m-1)>1+c$, contradicting part (iii) of the
definition of a $(c,D)$-bouquet. Putting together these last two proofs by
contradiction we deduce the claim.

Fixing $\b\in\nbq$, it similarly follows that if $\b'\in\nbq$ is equivalent
to $\b$, and $\b_n':[0,L_n']\to X$ as usual, then for each $m\in\N$ there
exists $t_0=t_0(m,\b')>0$ and $U_m\in\cU_m$ such that $\b_n'(t)\in U_m$
whenever $n\in\N$ and $t_0<t\le L_n'$. In fact if $K$ is the constant from
Definition~\ref{D:bouquet-bdy}, then we can take $t_0=m+3/2+K/2$.

Thus we get a well-defined natural map $\me:\bbd X\to\ebd X$ by taking
$\me(x)=(U_n)$ for $x=[(\b_n)]\in\bbd X$, where $U_m\in\cU_m$ for $m\in\N$
is defined by the requirement that $\b_n(t)\in U_m$ whenever $t$ is
sufficiently large, and $n$ is so large that $t\le L_n$.

Fixing $\e>0$, we now assume that every ball in $X$ has a finite $\e$-net,
and that $C$ is the rCAT(0) constant of $X$. We also fix a basepoint $o\in
X$. We claim that if $(x_n)$ is any sequence in $X$ such that $d(o,x_n)$ is
an increasing function of $n$ and $d(o,x_n)\ge n$, then there exists an
$(\e+C,D)$-bouquet $\g=(\g_n)$ such that the tips of $\g$ are located on
$1$-short paths from $o$ to some subsequence of $(x_n)$.

Let $D$ be a short function, and let us pick a sequence of $D$-short unit
speed paths $(\b_n)$ from $o$ to $x_n$. Writing $\b_{0,n}:=\b_n$, $n\in\N$,
the existence of a finite $\e$-net for the ball $\{x\in X\mid d(o,x)\le n\}$
allows us to pick a subsequence $(\b_{k,n})_{n=1}^\infty$ of
$(\b_{k-1,n})_{n=1}^\infty$ inductively so that the points $(\b_{k,n}(k))$
all lie in a ball of radius $\e$. We now take $\g_n:[0,n]\to X$ to be the
initial segment of $\b_{n,n}$ for each $n\in\N$. It follows readily from
Lemma~\ref{L:rough cx} that $(\g_n)_{n=1}^\infty$ is an $(\e+C,D)$-bouquet,
and so the claim follows.

We now apply the above procedure to each end $(U_n)$: just restrict the
initial sequence $(x_n)$ so that $x_n\in U_n$. Then $\g_n(n)$ is in the same
component of $X\setminus\overline{B(o,n-1)}$ as $x_{n,n}$, and it then
readily follows that $\me([\g])=(U_n)$. We have therefore proved (a).

Suppose $X$ is unbounded, so there exists a sequence $(x_n)$ in $X$ such
that $d(o,x_n)$ is an increasing function of $n$ and $d(o,x_n)\ge n$. By the
claim it follows that $\bbd X$ is nonempty, and so $\me(\bbd X)\subset\ebd
X$ is also nonempty. Conversely if $X$ is bounded, it is trivial that $\bbd
X$ and $\ebd X$ are empty. This concludes the proof of (b).
\end{proof}

\begin{rem}\label{R:totally bounded}
Using the concept of loose asymptoticity, the above proof can be adapted to
give the same conclusions if the $\e$-net assumption is replaced by the
following weaker one: for a sequence of values $r_n\to\infty$, the balls
$B(o,r_n)$ can be covered by a finite collection of balls of radius
$\d(r_n)$, where $\d$ is a little-o function. An example of a space
satisfying this assumption for every such sequence $(r_n)$ is the following
subspace $X$ of $l^2$:
$$ X = \{ x\in l^2\mid \|q_1(x)\|_2 \le \d(|p_1(x)|) \} $$
where $p_1$ is projection onto the first coordinate, $q_1=I-p_1$ is the
complementary projection, and $\d$ is a little-o function.
\end{rem}

\begin{rem}\label{R:end inj}
Although Theorem~\ref{T:end} tells us that $\me$ is surjective for many nice
spaces, $\me$ quite often fails to be injective in such cases. For instance
$\ibd\R^2=\bbd\R^2$ can be identified with the unit circle, but the
Euclidean plane $\R^2$ has only one end. To deduce that $\me:\bbd X\to\ebd
X$ is injective, it would suffice to assume some sort of mild bottleneck
condition in each end. In fact if $(U_n)$ is an end in an rCAT(0) space $X$,
and if there exists:
\begin{itemize}
\item a strictly increasing sequence of integers $(n_k)$,
\item a sequence of points $(x_k)$ with $x_k\in U_{n_k}$,
\item a little-o function $\d$, and
\item a sequence of positive numbers $(r_k)$ such that $r_k\le\d(n_k)$
    and such that every path from $o$ to $U_{n_k}\setminus B(x_k,r_k)$
    must pass through $B(x_k,r_k)$,
\end{itemize}
then there is a unique $x\in\bbd X$ such that $\me(x)=(U_n)$. For existence,
we show that a sequence of suitably short paths from $o$ to $x_n$, $n\in\N$,
gives a loose bouquet, and for uniqueness, we use Lemma~\ref{L:rough cx}. We
leave the details to the reader.
\end{rem}

\begin{rem}
However a bottleneck condition of the type considered in Remark~\ref{R:end
inj} is not necessary for $\me$ to be injective. To see this, we first
define
\begin{align*}
X_0 &:=  \{(x,y)\in\R^2\;:\; |y|\le x\le 1\}\,, \\
X_i &:=  \{(x,y)\in\R^2\;:\; |y|\le x,\; 2^{i-1}\le x\le 2^i\}\,,\qquad
i\in\N\,.
\end{align*}
We attach the Euclidean metric to each of the above sets. For all $i\ge 0$,
$X_i$ is a convex subset of the Euclidean plane, and so it is CAT(0). Now
let $X$ be the space obtained by isometrically gluing $X_{i-1}$ to $X_i$,
$i\in\N$, according to the following rule: if $i$ is odd, we glue along the
line segment of points $(2^{i-1},y)\in\R^2$, $0\le y\le 2^{i-1}$, while if
$i$ is even, we glue along the line segment of points $(2^{i-1},y)\in\R^2$,
$-2^{i-1}\le y\le 0$. Each gluing is along an isometric pair of closed
convex subsets of the complete spaces $X_{i-1}$ and $X_i$, and it follows as
in Example~\ref{X:CAT0 3} that $X$ is a complete CAT(0) space. It is also
clear that $X$ does not satisfy a bottleneck condition, that $\ibd X=\bbd X$
is a singleton set, and so $\me:\bbd X\to\ebd X$ is injective.
\end{rem}


\section{Sequential constructions}\label{S:seq}

In this section we show that in a Gromov hyperbolic length space, the
bouquet boundary can be naturally identified with the Gromov boundary. Since
the Gromov boundary is defined using sequences that ``march off to
infinity'', we embed this proof in a wider discussion of ways to define the
bouquet boundary in a general rCAT(0) space using such sequences.

As in the previous section, we can restrict these sequences in either a
tight or loose manner, depending on whether certain quantities are bounded
or grow more slowly than distance to the origin, and for the tight case we
can use a tight or loose equivalence. In this way we get three notions of
boundary at infinity. We will see that only two of them can be naturally
identified with the bouquet boundary, although all three of them can be
naturally identified with both the bouquet and the Gromov boundary in the
case of a Gromov hyperbolic length space. As in the previous section, we are
only talking about set theoretic identifications: we discuss an associated
topology in Section~\ref{S:top}.

We first record two lemmas. The first can be proved in the same manner as
its short arc variant in \cite[2.33]{Va}.

\begin{lem}\label{L:Vais}
Suppose that $\la$ is a $h$-short path from $x$ to $y$ in a metric space
$X$, and that $z\in X$. Then $\ip xyz\le\dist(z,\la)+h/2$. If additionally
$X$ is $\d$-hyperbolic, then $\dist(z,\la)\le\ip xyz+h+2\d$.
\end{lem}

The second lemma that we need is the so-called {\it Tripod Lemma} for
hyperbolic spaces. This version is as stated in \cite[2.15]{Va}, except that
again we are using short paths rather than short arcs.

\begin{lem}\label{L:Tripod}
Suppose that $\g_1$ and $\g_2$ are unit speed $h$-short paths from $o$ to
$x_1$ and $x_2$, respectively, in a $\d$-hyperbolic space. Let $u_1=\g_1(t)$
and $u_2=\g_2(t)$ for some $t\ge 0$, where $d(o,u_1)\le\IP<x_1,x_2;o>$. Then
$d(u_1,u_2)\le 4\d+2h$.
\end{lem}

As defined in Definition~\ref{D:Gromov seq}, the notion of a sequence that
``marches off to infinity'' is played by a Gromov sequence, defined as a
sequence $(x_n)$ such that $\IP<x_m,x_n;o>\to\infty$ as $m,n\to\infty$, and
the Gromov boundary $\Gbd X$ is defined by an associated equivalence
relation. This definition is not however consistent with the bouquet
boundary; see Proposition~\ref{P:Gb bad1} below. Instead we proceed as
follows.

If $\b$ is a loose $(\d,D)$-bouquet with initial point $o$ in an rCAT(0)
space $X$, and $x=f(\b)$ is the sequence of tips of $\b$ (so
$x=(\b_n(L_n))$), then it follows from Lemma~\ref{L:Vais} that
$$
\IP<o,x_n;x_m>\le\d'(d(o,x_m))\,, \qquad
  \text{for all } m,n\in\N,\; m\le n\,.
$$
where $\d'(t)=\d(t)+1+1/2$. Notice that the term $1$ in $\d'(t)$ bounds the
difference $|\d(L_m)-\d(d(o,x_m))|\le|L_m-d(o,x_m)|$, where we use the
Lipschitz property of $\d$; we make similar estimates in future without
comment. The term $1/2$ is an upper bound for $D(d(o,x_n))$; it could of
course be replaced by $D(d(o,x_1))/2\le 1/2$.

\begin{dfn} Given a little-o function $\d$, a {\it loose $\d$-bouquet
sequence (with basepoint $o$)} is a sequence $(x_n)_{n=1}^\infty$ in an
rCAT(0) space $X$ such that $(d(o,x_n))_{n=1}^\infty$ is an unbounded
monotonically increasing sequence, and such that
$$
\IP<o,x_n;x_m>\le\d(d(o,x_m))\,, \qquad
  \text{for all } m,n\in\N,\; m\le n\,.
$$
As before, we define a {\it $c$-bouquet sequence (with basepoint $o$)} as
above but with $\d$ equal to some constant function $c\ge 0$. We denote by
$\lbqs$ and $\nbqs$ the sets of loose bouquet sequences and bouquet
sequences, respectively, in both cases with basepoint $o$; we omit the
basepoint from this notation except in Observation~\ref{O:bs oo'} where we
note that these notions are essentially independent of the basepoint.
\end{dfn}

By the discussion before the above definition, we see that if $\b$ is a
loose $(\d,D)$-bouquet with initial point $o$ in an rCAT(0) space $X$, and
$x=f(\b)$ is the sequence of tips of $\b$ (so $x=(\b_n(L_n))$), then $x$ is
a loose $\d'$-bouquet sequence with basepoint $o$, with $\d'$ as above. We
call $f:\lbq\to\lbqs$ the {\it tip map} of $X$ (at basepoint $o$); $f$ also
maps $\nbq$ to $\nbqs$.

We now define the associated notions of boundary at infinity in the natural
way. Although these notions are defined in terms of a basepoint $o$, they
will turn out to be independent of $o$; see Observation~\ref{O:bs oo'}.

\begin{dfn}\label{D:l bouquet seq bdy}
Fixing a basepoint $o\in X$, we define the {\it loose bouquet sequence
boundary of $X$}, $\lbsbd X$, to be the set of equivalence classes given by
all loosely asymptotic bouquet sequences with basepoint $o$, where two loose
bouquet sequences $x=(x_n)$ and $y=(y_n)$ are {\it loosely asymptotic},
written $x\sim_{\rLS} y$, if there exists a little-o function $\d$ such that
$$
\IP<o,x_m;y_n>\mino\IP<o,y_n;x_m> \le \d(d(o,x_m)\mino d(o,y_n))\,,
  \qquad m,n\in\N\,,
$$
We denote by $[x]_{\rLS}$ the equivalence class in $\lbsbd X$ containing a
given loose bouquet sequence $x$.
\end{dfn}

\begin{dfn}\label{D:bouquet seq bdy}
Similarly, we define the {\it bouquet sequence boundary of $X$}, $\bsbd X$,
to be the set of equivalence classes given by all asymptotic bouquet
sequences with basepoint $o$, where two bouquet sequences $x=(x_n)$ and
$y=(y_n)$ are {\it asymptotic}, written $x\sim_{\rS} y$, if they are loosely
asymptotic for $\d$ equal to some constant function $K\ge 0$. We denote by
$[x]_\rS$ the equivalence class in $\bsbd X$ containing a given bouquet
sequence $x$.
\end{dfn}

We now make a couple of observations about ways of constructing new (loose)
bouquet sequences that are equivalent to a given bouquet sequence.

\begin{obs}\label{O:bs subs}
Any subsequence of a (loose) bouquet sequence $x$ is a (loose) bouquet
sequence that is (loosely) asymptotic to $x$; we refer to such a subsequence
as a {\it (loose) bouquet subsequence}.
\end{obs}

\begin{obs}\label{O:bs oo'}
A $c$-bouquet sequence $x=(x_n)$ with basepoint $o$ may not be a bouquet
sequence for another basepoint $o'$, since $d(o',x_n)$ might not be
increasing. However, by thinning out $x$ to get a subsequence $y=(y_n)$
where $d(o,y_{n+1})\ge d(o,y_n)+2d(o,o')$, it follows easily from the
triangle inequality that $y$ is a $c'$-bouquet sequence with basepoint $o'$,
where $c'=c+d(o,o')$, and that $y\sim_\rS x$. In a similar fashion, by
taking a suitable subsequence of $x\in\lbqso{o}$, we get a sequence
$y\in\lbqso{o'}$, with $y\sim_\rLS x$. Thus $\bsbd X$ and $\lbsbd X$ are
independent of the basepoint $o$.
\end{obs}

\begin{thm}\label{T:bouq bij 2}
Suppose $X$ is an rCAT(0) space. Then the tip map $f:\lbq\to\lbqs$ induces a
natural bijection $\is:\lbbd X\to\lbsbd X$. Consequently, $\lbsbd X$ is
naturally bijective with $\bbd X$, and also with $\ibd X$ if $X$ is CAT(0).
\end{thm}

\begin{proof}
The first statement of our theorem follows from the claim that
$\is([\b]_{\rL})=[x]_{\rLS}$ is a well-defined map, where $\b=(\b_n)\in\lbq$
and $x=(x_n)=f(\b)$. Because of Observations \ref{O:b subs} and \ref{O:bs
subs}, it suffices to prove the claim after taking any desired subsequence,
so we may assume that the length of $\b_n$ grows as fast as desired.

To establish the claim, we suppose $\b^1=(\b_n^1)$ and $\b^2=(\b_n^2)$ are
loosely asymptotic loose $(\d,D)$-bouquets, with $\d_1$ being the little-o
function in Definition~\ref{D:loose-bouquet-bdy} for this pair of loose
bouquets, and we denote the length of $\b_n^i$ by $L_n^i$, $i=1,2$. Let
$x=(x_n)=f(\b^1)$ and $y=(y_n)=f(\b^2)$. Let $o$ and $o'$ be the initial
points of $\b^1$ and $\b^2$, respectively. We know that $x$ and $y$ are
loose bouquet sequences with basepoints $o$ and $o'$, respectively, and so
both are loose bouquet sequences from $o$ once we suitably thin out $\b^2$
(and so $y$), as we may.

Fixing arbitrary $m,n\in\N$, assume first that $d(o,x_m)\le d(o,y_n)$. Now
$$
\dist(x_m,\b_n^2) \le \dist(x_m,\b_n^2(L_m^1)) \le \d(L_m^1)\,.
$$
By Lemma~\ref{L:Vais}, we deduce that $\IP<o',y_n;x_m>\le \d(L_m^1)+1/2$,
and so
$$
\IP<o,y_n;x_m>\le \d(L_m^1)+d(o,o')+1/2 \le \d(d(o,x_m))+d(o,o')+3/2\,.
$$
If instead $d(o,x_m)> d(o,y_n)$, we similarly get
$$ \IP<o,x_m;y_n> \le \d(d(o,y_n))+3/2\,. $$
Since $m,n$ are arbitrary, it follows that $x$ and $y$ are loosely
asymptotic and the claim is established.

We next show that $\is$ is injective. Suppose
$\is([\b^1]_{\rL})=\is([\b^2]_{\rL})$, where $x=(x_n)$, $y=(y_n)$, and
$(L_n^i)$, $i=1,2$, are as before, except now instead of assuming that
$\b^1$ and $\b^2$ are loosely asymptotic, we want to prove this. By taking
subsequences if necessary, we assume that
$$
L_{n+1}^1 > L_n^2+2+\d(L_n^2) > L_n^1+4+\d(L_n^2)+\d(L_n^1)\,,
  \qquad n\in\N\,,
$$
where $\d$ is the loose asymptoticity parameter for $x,y$ in
Definition~\ref{D:l bouquet seq bdy}. By the triangle inequality, this last
pair of inequalities ensures that $\IP<o,x_m;y_n> > \d(L_m^1)$ whenever
$n\ge m$ and $\IP<o,y_n;x_m>
> \d(L_n^2)$ whenever $m>n$, so loose asymptoticity of $x$ and $y$
tells us that $\IP<o,y_n;x_m> \le \d(L_m^1)$ whenever $n\ge m$ and
$\IP<o,x_m;y_n> \le \d(L_n^2)$ otherwise.

To show that $\b^1,\b^2$ are loosely asymptotic, we bound
$d(\b_m^1(t),\b_n^2(t))$ by a suitably little-o quantity for $m\le n$ and
$t\le L_m^1$; the corresponding result when $m>n$ is handled similarly. It
suffices to establish such a bound when $t=L_m^1$ (and so $\b_m^1=x_m$),
since a corresponding bound for smaller $t$ follows from Lemma~\ref{L:rough
cx}. Without loss of generality, we assume that $d(o,x_m)\ge 1$ and
$d(o,y_n)\ge 1$.

Since $\IP<o,y_n;x_m>\le \d(L_m^1)$, the concatenation of $\b_m^1$ and a
$D$-short path $\g$ from $x_m$ to $y_n$ gives a $K$-short path from $o$ to
$y_n$, where $K:=2\d(L_m^1)+2$. Inequality \eqref{E:wrCAT0} says that
$$
(d(x_m,u)-C)^2 \le (1-t)(d(x_m,o))^2+t(d(x_m,y_n))^2-t(1-t)(d(o,y_n))^2\,,
$$
where $u:=\b_n^2(s)$, $s:=L_m^1$, and $0\le t\le 1$ is any number satisfying
$td(o,y_n)\le s$ and $(1-t)d(o,y_n)\le L-s$. This last pair of inequalities
holds for $t=s/L$, where $L:=L_n^2$. Writing $\D:=d(o,y_n)$, we note that
$d(x_m,o)\le s$ and $d(x_m,y_n)\le \len(\g)\le d(o,y)+K-s\le L+K-s$, so
\begin{equation}\label{E:prelim}
(d(x_m,u)-C)^2 \le s^2(1-s/L)+(s/L)(L-s+K)^2-t(1-t)\D^2\,.
\end{equation}
Contrasting the inequalities $t\D\le s$ and $(1-t)\D\le L-s$ with $\D\ge
L-1/\D$, we see that we can almost reverse the first two inequalities:
$t\D\ge s-1/\D$ and $(1-t)\D\ge L-s-1/\D$. Thus
$$ t(1-t)\D^2 \ge (s-1/\D)(L-s-1/\D) \ge s(L-s) - 2\,. $$
Since also $s^2(1-s/L)+(s/L)(L-s)^2 = s(L-s)$, it follows from
\eqref{E:prelim} that
$$
(d(x_m,u)-C)^2 \le 2 + s(2(L-s)K+K^2)/L \le  2 + 2sK + K^2\,.
$$
From this it readily follows that $d(x_m,u)$ is a little-o function of $s$,
as required.

Next we prove that $\is$ is surjective. Fix a loose bouquet sequence
$x=(x_n)$. By taking a subsequence if necessary, we may assume that
$d(o,x_{n+1})\ge d(o,x_n)+1$ for all $n$. We now construct unit speed
$D$-short paths $\b_n$ from $o$ to $x_n$ for each $n$ and claim that the
resulting sequence $\b=(\b_n)$ is a loose bouquet such that
$\is([\b]_{\rL})=[x]_{\rL}$. The only non-trivial part of the proof is that
$d(\b_m(t),\b_n(t))$ is dominated by some little-o function of $t$. The
proof of this goes along the same lines as the proof of injectivity, so we
leave it to the reader.

The final statement in the theorem follows from Corollary~\ref{C:5 equiv}
and Theorem~\ref{T:CAT0 bij}.
\end{proof}

\begin{rem}\label{R:tip}
A fact that will be useful in the next section is that the restricted tip
map $g:=f|_{\sbq}$ induces the bijection $\is:\lbbd X\to\lbsbd X$. This
follows from Corollary~\ref{C:5 equiv} and the proof of Theorem~\ref{T:bouq
bij 2}. Thus the following diagram commutes, where surjectivity is as usual
denoted by a double arrow.
$$
\xymatrix{ \sbq \ar@{->>}[d] \ar@{->>}[r]^g & g(\sbq) \ar@{->>}[d]
\ar@{^{(}->}[r]
  & \lbqs \ar@{->>}[dl] \\
\lbbd X \ar@{^{(}->>}[r]^{\is} & \lbsbd X
}
$$
\end{rem}

\begin{thm}\label{T:bouq surj}
Suppose $X$ is an rCAT(0) space. Then the map of a bouquet sequence to
itself defines a natural surjection $\ms:\bsbd X\to\lbsbd X$.
\end{thm}

\begin{proof}
The fact that $\ms$ is well-defined is trivial. As for surjectivity, suppose
that $x=(x_n)$ is a loose $\d$-bouquet sequence. By taking a subsequence if
necessary, we may assume that $d(o,x_{n+1})\ge d(o,x_n)+1$ for all $n$. By
the proof of Theorem~\ref{T:bouq bij 2}, $x$ is the sequence of tips of a
loose $(\d',D)$-bouquet $\b$ (for any short function $D$ of our choice, and
some $\d'$ dependent only on $\d$ and $D$). Let $\b'=(\b_n')$ be a bouquet
obtained by pruning $\b$, as in the surjectivity part of the proof of
Theorem~\ref{T:bouq bij 1}; we denote by $x'=(x_n')$ the associated sequence
of tips, and let $K$ be the constant of asymptoticity as in
Definition~\ref{D:bouquet-bdy}. Now $\b$ is loosely asymptotic to $\b'$, so
the associated sequences of tips $x$ and $x'$ are loosely asymptotic by the
proof that $\is$ is well-defined in Theorem~\ref{T:bouq bij 2}.
\end{proof}

\begin{rem}\label{R:bs lbs}
We have proved in the above theorem that the following diagram of natural
maps commutes, with injections and surjections as indicated. Note that the
vertical maps are quotient maps, and that the map from $\nbqs$ to $\lbqs$ is
the identity map. Note also that the natural (composition) map from $\nbqs$
to $\lbsbd X$ is surjective.
$$
\xymatrix{ \nbqs \ar@{->>}[d] \ar@{^{(}->}[r] & \lbqs \ar@{->>}[d] \\
\bsbd X \ar@{->>}[r]^{\ms} & \lbsbd X
}
$$
\end{rem}

The following examples show that, unlike the maps $\il$ and $\is$ in
Theorems~\ref{T:bouq bij 1} and \ref{T:bouq bij 2}, the map $\ms$ in
Theorem~\ref{T:bouq surj} is not necessarily an injection.

\begin{exa}\label{X:bouq surj 1}
Consider the complex conjugate sequences $x=(x_n)$ and $y=(y_n)$ in the
complex plane $\C$ given by $x_n=4^n+2^ni$, $y_n=4^n-2^ni$, $n\in\N$, where
$i=\sqrt{-1}$. Suppose $m\le n$, $m,n\in\N$. The simple estimate
$\sqrt{1+t}\le 1+t$ when $t>0$ yields
$$
d(o,x_m) = \sqrt{2^{4m}+2^{2m}} = 2^{2m}\sqrt{1+2^{-2m}} \le 2^{2m}+1\,,
$$
and similarly
$$ d(x_m,x_n) \le 2^{2n}-2^{2m}+1\,. $$
Thus
$$
2\IP<o,x_n;x_m> \le (2^{2m}+1) + (2^{2n}-2^{2m}+1) - 2^{2n} = 2\,,
$$
and so $x$ is a bouquet sequence. By symmetry, $y$ is also a bouquet
sequence. It is readily verified that $x$ and $y$ are not equivalent as
bouquet sequences, but that they are equivalent as loose bouquet sequences.
\end{exa}

\begin{exa}\label{X:bouq surj 2}
Our second example is a variant of Example~\ref{X:bouq surj 1} with $x_n^t =
4^n+2^n t i$, where $-1 \le t\le 1$ and again $i=\sqrt{-1}$. For each $t$,
$x^t=(x_n^t)$ is a bouquet sequence in the convex subset $X$ of the complex
plane consisting of all $x+yi$ with $0\le x<\infty$ and $y^2\le x$. If $s\ne
t$ then $x^s$ and $x^t$ are inequivalent as bouquet sequences (although they
are loosely equivalent). It follows that $\bsbd X$ has the cardinality of
the continuum even though $\bbd X$ is a singleton set. Thus the map $\ms$ is
far from being injective in this case.
\end{exa}

Although $\bsbd X$ and $\lbsbd X$ are in general different, we now show that
both can be naturally identified with the Gromov boundary $\Gbd X$ of a
Gromov hyperbolic space $X$.

\begin{thm}\label{T:Gromov bij}
Suppose $X$ is an rCAT(0) space. Then every loose bouquet sequence is a
Gromov sequence, and this identity map induces natural maps $\mm:\lbsbd
X\to\Gbd X$ and $\mml=\mm\circ\ms:\bsbd X\to\Gbd X$. If $X$ is Gromov
hyperbolic, then $\mm$ and $\ms$ (and hence $\mml$) are bijective.
\end{thm}

\begin{proof}
Suppose $(x_n)$ is a (loose) bouquet sequence. By direct computation,
\begin{equation}\label{E:gp eqn}
d(o,x_m) = \IP<x_m,x_n;o> + \IP<o,x_n;x_m>.
\end{equation}
Taking $n\ge m$ and letting $m\to\infty$, the left-hand side of the above
equation tends to infinity faster than $\IP<o,x_n;x_m>=\d(d(o,x_m))$, and so
$\IP<x_m,x_n;o>$ must tend to infinity as $m,n\to\infty$. Thus $(x_n)$ is a
Gromov sequence.

In a similar way, it follows that any pair of (loosely) asymptotic bouquet
sequences must be equivalent as Gromov sequences. Thus the identity map
gives rise to well-defined natural maps $\mm:\lbsbd X\to\Gbd X$ and
$\mml:\bsbd X\to\Gbd X$. Since $\ms:\bsbd X\to\lbsbd X$ is also induced by
an identity map on bouquet sequences, it follows that $\mml=\mm\circ\ms$. It
remains to prove that these maps are bijective under the added assumption
that $X$ is Gromov hyperbolic.

We already know that $\ms$ is surjective (Theorem~\ref{T:bouq surj}), and we
now prove that $\mm$ is surjective. Given a Gromov sequence $(x_n)$, we thin
it out if necessary to ensure that
$$ \IP<x_m,x_n;o>\ge m\mino n\,, \qquad m,n\in\N\,. $$
Note that a subsequence of a Gromov sequence is always an equivalent Gromov
sequence. In particular, $d(x_n,o)=\IP<x_n,x_n;o>\ge n$, $n\in\N$. We may
also assume that $d(x_n,o)$ is an increasing function of $n$. For each
$n\in\N$, let $\la_n:[0,L_n]\to X$ be a unit speed $1$-short path from $o$
to $x_n$, let $\g_n$ be the initial segment of $\la_n$ of length $n$, and
let $y_n=\g_n(n)$. As before,
$$ d(o,y_n) = \IP<y_n,x_n;o> + \IP<o,x_n;y_n>. $$
Since $\la_n$ is a $1$-short segment, we must have $\IP<o,x_n;y_n>\le 1/2$,
and so $\IP<y_n,x_n;o>\to\infty$ as $n\to\infty$. A repeated application of
hyperbolicity gives
$$
\IP<y_m,y_n;o> \ge
  \IP<y_m,x_m;o>\mino\IP<x_m,x_n;o>\mino\IP<x_n,y_n;o>-2c\,,
$$
so $\IP<y_m,y_n;o>\to\infty$ as $m,n\to\infty$. Thus $(y_n)$ is a Gromov
sequence and it is equivalent to $(x_n)$.

Suppose now that $m,n\in\N$, $m<n$, and let $z_m=\g_n(m)$. Then $d(y_m,o)\le
m\le\IP<x_m,x_n;o>$, so it follows from the Tripod Lemma
(Lemma~\ref{L:Tripod}) that $d(y_m,z_m)\le 4c+2$. Consequently,
\begin{align*}
2\IP<o,y_n;y_m> &= d(o,y_m)+d(y_m,y_n)-d(o,y_n) \\
&\le m + (4c+2+n-m) - (n-1) = 4c+3\,,
\end{align*}
and so $(y_n)$ is a bouquet sequence. Thus $\mm$ is surjective as required.

To prove injectivity of $\mm$ and $\ms$, it suffices to show that if
$x=(x_n)$ and $y=(y_n)$ are non-asymptotic $c$-bouquet sequences, then $x$
are $y$ are not equivalent as Gromov sequences. Since $x$ and $y$ are
non-asymptotic, we can fix indices $M,N\in\N$ such that
\begin{equation}\label{E:non-asymp1}
J:=\IP<o,x_M;y_N> \mino \IP<o,y_N;x_M> > 2c+6\d\,.
\end{equation}
We claim that
$$
\IP<x_m,y_n;0>\le (d(o,x_M)\maxo d(o,y_N))+c+2\d\,,
  \qquad\text{for all } m>M,\,n>N.
$$
Assuming this claim, we see that $x':=(x_{i+M})_{i=1}^\infty$ is not
equivalent to $y':=(y_{i+N})_{i=1}^\infty$. Since $x$ is equivalent to $x'$,
and $y$ to $y'$, it follows that $x$ is not equivalent to $y$, as required.

Let us prove the claim. Suppose $m>M$ and $n>N$. Fixing an arbitrary $h>0$,
we choose a $h$-short path $\g_1:[0,L_1]\to X$ from $o$ to $x_m$
parametrized by arclength. By Lemma~\ref{L:Vais}, $d(x_M,\g_1)\le
c_0:=c+2\d+h$, so let $u_1:=\g_1(t_1)$, $0\le t_1\le L_1$, be such that
$d(x_M,u_1)\le c_0$. Similarly, we choose a $h$-short path $\g_2:[0,L_2]\to
X$ from $o$ to $y_n$ parametrized by arclength, and then there exists
$v_2:=\g_2(t_2)$, $0\le t_2\le L_2$, such that $d(y_N,v_2)\le c_0$. It
follows that
\begin{equation}\label{E:u1v2}
\IP<o,u_1;v_2> \mino \IP<o,v_2;u_1> > J-2c_0\,.
\end{equation}
Without loss of generality, we assume that $t_1\le t_2$. Writing
$u_2=\g_2(t_1)$, we see
$$ d(u_1,v_2)-d(u_1,u_2)\le d(u_2,v_2)\le t_2-t_1\,, $$
and, by the shortness of $\g_2$,
$$ t_2-t_1-h \le d(o,v_2)-d(o,u_2)\,. $$
It follows that $\IP<o,u_2;u_1> > \IP<o,v_2;u_1> - h/2$, and so
\eqref{E:u1v2} implies that if $h>0$ is sufficiently small, then
$$ \IP<o,u_2;u_1> > J-2c_0-h/2 > 2\d+h\,. $$
By $h$-shortness we have
$$ |\IP<o,u_2;u_1>-\IP<o,u_1;u_2>|\le |d(o,u_1)-d(o,u_2)|\le h $$
and so again if $h>0$ is sufficiently small, then
$$ \IP<o,u_1;u_2> > J-2c_0-3h/2 > 2\d+h\,. $$
Thus for $h>0$ sufficiently small, we have
$$ d(u_1,u_2) = \IP<o,u_2;u_1>+\IP<o,u_1;u_2> > 4\d+2h\,. $$
In view of Lemma~\ref{L:Tripod}, we conclude that
$$ \IP<x_m,y_n;o> < d(o,u_1)\le d(0,x_M)+c_0=d(o,x_M)+c+2\d+h\,. $$
Since $h>0$ is arbitrary, the claim follows.
\end{proof}

The following result shows that the natural maps $\mm$ and $\mml$ may fail
to be injective if $X$ is not Gromov hyperbolic. These maps can also fail to
be surjective in complete CAT(0) spaces according to \cite[Theorem~1]{BF2}.

\begin{prop}\label{P:Gb bad1}
Suppose first that $X=\R^n$ for $n>1$ with the Euclidean metric attached.
Then $\bbd X=\lbsbd X$ has the cardinality of the continuum, while $\Gbd X$
is a singleton set. The natural maps $\mm$ and $\mml$ are not injective.
\end{prop}

\begin{proof}
Since $X$ is a complete CAT(0) space, $(\ibd X,\ct)$ is homeomorphic to the
sphere $S:=\partial B(0,1)$, and so its cardinality is that of the
continuum. The same is true of $\bbd X=\lbsbd X$ by Theorem~\ref{T:CAT0
bij}. We claim that $\Gbd X$ is a singleton set. Assuming this claim, it is
clear that $\mm$ and $\mml$ cannot be injective.

It remains to justify our claim. Certainly $\Gbd X$ is nonempty because of
the natural map $\mm$. We now appeal to Theorem~2.2 of \cite{BK} which
states that $\mm$ is surjective if $X$ is a proper geodesic space. In fact,
as is clear from the proof of that result, $\mm$ is induced by the map that
takes a geodesic ray $\g:[0,\infty)\to X$ parametrized by arclength to the
Gromov sequence $(\g(t_n))_{n=1}^\infty$, where $(t_n)$ is any sequence of
non-negative numbers with limit infinity. Since the ideal boundary of a
complete CAT(0) space can be viewed as the set of geodesic rays from any
fixed origin, it follows that we get representatives of all points in
$\Gbd\R^2$ by considering only the Gromov sequences
$x^t:=(na_t)_{n=1}^\infty$, where $a_t=(\cos t,\sin t)\in\R^2$, $ t\in\R$. A
straightforward calculation shows that $(x^t,x^s)\in E$ for all pairs $t,s$,
except when $|t-s|$ is an odd multiple of $\pi$, i.e.~except when $x^t$ and
$x^s$ are tending to infinity in opposite directions. But in the exceptional
case, we have $(x^t,x^{t+\pi/2})\in E$ and $(x^{t+\pi/2},x^s)\in E$, so all
Gromov sequences are equivalent, and we have proved our claim.
\end{proof}

Finally, we relate the Gromov and end boundaries. As in
Proposition~\ref{P:Gb bad1}, we view all our varieties of the bouquet
boundary as being the same. If we do not make this identification, then the
second statement of this result should instead state that
$\me=\mf\circ\mm\circ\is\circ\il$, where $\il$ and $\is$ are as in Theorems
\ref{T:bouq bij 1} and \ref{T:bouq bij 2}, respectively.

\begin{prop}\label{P:Gromov end map}
Suppose $X$ is a metric space. Then there is a natural map $\mf$ from $\Gbd
X$ to $\ebd X$. Furthermore $\me=\mf\circ\mm$, with $\me$ as in
Theorem~\ref{T:end} and $\mm$ as in Theorem~\ref{T:Gromov bij}.
\end{prop}

\begin{proof}
Suppose $x=(x_n)$ is a Gromov sequence, and let $o$ be the basepoint for
$\Gbd X$ and $\ebd X$, as usual. Let $f(n)$ be the smallest $k\in\N$ such
that $\IP<x_i,x_j;o> > n$ for all $i,j\ge k$, so that $f$ is a monotonically
increasing sequence with limit infinity (because $x$ is a Gromov sequence).
Note in particular that $d(x_i,o)=\IP<x_i,x_i;o> > n$ when $i\ge f(n)$.

Let $U_n$ be the component of $X\setminus \overline{B(o,n)}$ containing
$x_{f(n)}$. We claim that $(U_n)$ is an end. To show this, it suffices to
show that $x_m\in U_n$ for all $m\ge f(n)$. If this were not true, then any
path from $x_m$ to $x_{f(n)}$ would have to pass through $B(o,n)$, and so
$d(x_m,x_{f(n)})$ would be larger than $(d(x_m,o)-n)+(d(x_{f(n)},o)-n)$,
which would imply that $\IP<x_m,x_{f(n)};o> < n$, in contradiction to our
construction. Thus we have a map from Gromov sequences to ends, and we can
see in a similar fashion that if $x,y$ are two Gromov sequences with
$(x,y)\in E$, then this map takes them to the same end. It follows that this
map induces a natural map $\mf:\Gbd X\to\ebd X$.

The last statement in the theorem follows easy because $\me$ and $\mf$ are
both induced by set containment, and $\mf$ by the tip map (or just the
identity map, if we view the bouquet boundary as being given by the $\lbsbd
X$ variant).
\end{proof}

Note that the $\mf:\Gbd X\to\ebd X$ need not be injective even if $X$ is a
complete CAT(-1) space (and so both CAT(0) and Gromov hyperbolic), as
evidenced by the hyperbolic plane. Also $\mf$ need not be surjective among
complete CAT(-1) spaces \cite[Theorem~2]{BF2}.

In summary, by putting together Corollary~\ref{C:5 equiv},
Theorem~\ref{T:bouq bij 2}, and Remark~\ref{R:bs lbs}, we see that sets of
equivalence classes as listed below lead to seven naturally equivalent
notions of boundary at infinity for an rCAT(0) space $X$, and we call them
all the bouquet boundary. (By ``naturally equivalent'', we mean that there
is a natural bijection.)
\begin{itemize}
\item Asymptotic (bouquets or standard bouquets).
\item Loosely asymptotic (loose bouquets, bouquets, or standard
    bouquets).
\item Loosely asymptotic (bouquet sequences or loose bouquet sequences).
\end{itemize}

Furthermore, we have the following commutative diagram of natural maps
between the various types of boundaries that we have considered:
$$
\xymatrix{
  &&& \bsbd X \ar@{->>}[d]^{\ds{\ms}} \\
  \ibd X\; \ar@{^{(}->}[r]^{\ds{\ii}} &
    \;\bbd X\; \ar@{->}[d]^{\ds{\me}} \ar@{=}[r]^{\ds{\il}} &
    \;\lbbd\; X \ar@{=}[r]^{\ds{\is}} & \;\lbsbd X \ar@{->}[d]^{\ds{\mm}}\\
  & \ebd X \ar@{<-}[rr]_{\ds{\mf}} && \Gbd X\\
}
$$
Here $\ii$ is injective, or bijective if $X$ is complete CAT(0)
(Theorem~\ref{T:CAT0 bij}). For the bijections $\il$ and $\is$, see Theorems
\ref{T:bouq bij 1} and \ref{T:bouq bij 2}, respectively. $\ms$ is
surjective, or bijective if $X$ is Gromov hyperbolic, in which case $\mm$ is
also bijective (Theorems~\ref{T:bouq surj} and \ref{T:Gromov bij}).
Conditions for $\me$ to be surjective or injective are given in
Theorem~\ref{T:end} and Remark~\ref{R:end inj}, respectively.


\section{The bouquet topology}\label{S:top}

In this section, we define a bouquet topology $\bt$ on $\bcl X:=X\cup \bbd
X$ which makes $\bcl X$ into a {\it bordification} of $X$, i.e.~$X$ with its
metric topology is a dense subspace of $(\bcl X,\bt)$.

Throughout this section, we assume implicitly that $X$ is a $C$-rCAT(0)
space and $\bbd X$ is nonempty (and so $X$ is unbounded). The origin $o\in
X$ is fixed but arbitrary. In all cases $D$ and $L_n$ are as defined for
standard bouquets, i.e.~$D(t):=1/(1\maxo(2t))$ for all $t\ge 0$, and
$L_n:=(2C+2)^n$. For convenience, $L_0:=1$.

In order to proceed, we define a version of $X$ that is similar to the
definition of $\bbd X$, i.e.~we view $X$ as a set of equivalence classes of
objects vaguely resembling standard bouquets.

\begin{dfn}\label{D:fl bouquet}
A {\it mother bouquet from $o\in X$ to $x\in X$} is simply a $D$-short unit
speed path $\g:[0,L]\to X$ from $o$ to $x$. It is convenient to define the
associated {\it finite length bouquet $\b:=(\b_n)$ from $o$ to $x$} by
$\b_n=\tg|_{[0,L_n]}$, where $\tg$ is defined by $\tg|_{[0,L]}=\g$ and
$\tg(t)=x$ for all $t>L$. We also call $\b$ the {\it child of $\g$}, and
$\b_n(L_n)$, $n\in\N$, the {\it tips of $\b$}. We write $L_n':=L_n\mino L$,
so that $\b_n|_{[0,L_n']}$ is always a unit speed segment.
\end{dfn}

Note that if $\g:[0,L]\to X$ is a mother bouquet from $o$ to $x$, then
$L-d(o,x)\le D(d(0,x))\le 1$. Thus $\tg(t)=x$ for all $t\ge d(o,x)+1$. We
say that all finite length bouquets from $o$ to $x\in X$ are {\it
destination equivalent} and denote this equivalence class by $i(x)$. Thus
$i:X\to i(X)$ is a bijection.

Fixing an rCAT(0) space $X$, we denote by $\gbq$ the set of all {\it
generalized bouquets from $o$}, meaning the set of all standard and finite
length bouquets from $o$. Identifying $X$ with $i(X)$, we view $\bcl X$ as a
set of equivalence classes of generalized bouquets, where these classes are
defined using destination equivalence for finite length bouquets and
asymptoticity for standard bouquets.

Define the product topological space
$$
P:=\prod_{n=1}^\infty\,
  \prod_{0\le t\le n\vphantom{{}_{\displaystyle A}}} X_n^t
$$
where $X_n^t$ is the closed ball of all $x\in X$ such that $d(o,x)\le
L_n+1$; the $t$-superscript serves only to distinguish between copies of
this ball. Denote by $p_{nt}:P\to X_n^t$ the associated projection maps.

Because the paths $\b_n$ of $\b\in\gbq$ are all of subunit speed, we see
that $\b_n(t)\in X_n^t$ and so $\b$ can naturally be viewed as an element of
$P$. Let $q:\gbq\to\bcl X$ be the quotient map consistent with our
definition of $\bcl X$, i.e.~$q(\b)=q(\b')$ if and only if $\b,\b'$ are
either asymptotic standard bouquets or destination equivalent finite length
bouquets. $P$ induces a subspace topology on $\gbq$, and $\bcl X$ then
receives the quotient topology for $q$.
$$
\xymatrix{ \gbq \ar@{->>}[d]_-q \ar@{^{(}->}[r] & P \ar@{->>}[d]^{p_{nt}}\\
\bcl X & X_n^t }
$$

\begin{dfn}\label{D:bouquet top}
The {\it bouquet topology} $\bt$ on $\bcl X$ is the quotient topology for
$q$. Henceforth, $\bcl X$ and $\bbd X$ are shorthand for $(\bcl X,\bt)$ and
its subspace $(\bbd X,(\bt)_{\bbd X})$. We call $(\bt)_{\bbd X}$ the bouquet
topology on $\bbd X$ and also denote it simply as $\bt$.
\end{dfn}

In order to give alternative, more explicit, definitions of the bouquet
topology, we first define some sets containing elements of $\bcl X$ that are
somehow close to $x\in\bbd X$. In these definitions, which we use throughout
this section, we assume that $r>0$, $n\in\N$, and $0\le t\le L_n$.
\begin{align*}
S'(x,r;n,t) &= \{ y\in \bcl X\mid%
  \forall\,\b\in q^{-1}(x),\b'\in q^{-1}(y):
               d(p_{nt}(\b),p_{nt}(\b')) < r \},\\
S(x,r;n,t) &= \{ y\in \bcl X\mid
  \exists\,\b\in q^{-1}(x),\b'\in q^{-1}(y):
               d(p_{nt}(\b),p_{nt}(\b')) < r \},\\
S_0(x;n,t) &= \{ y\in \bcl X\mid
  \exists\,\b\in q^{-1}(x),\b'\in q^{-1}(y):
               p_{nt}(\b) = p_{nt}(\b') \},\\
I(x;n) & = \bigcap_{\substack{1\le m\le n\\ 0\le t\le L_n}} S_0(x;n,t).
\end{align*}

In Figure~\ref{f:S nbd}, we give a rough illustration of what $S(x,r;n,t)$
looks like in one particular instance. Here $X$ is the Euclidean plane,
$n=2$, and $q(\b)=x$. Let us assume that $y\in S(x,r;2,t)$ with $y=q(\b')$,
and that $t\ge r\ge 2$. The shortness parameter of $D_2:=D(L_2)$ is fairly
small: certainly $D_2\le 1/8$, and $D_2$ is much smaller than this if $C$ is
large, so the constraint $d(\b_2(t),\b_2'(t))<r$ forces $d(o,\b_2'(t))$ to
lie in the interval $[t-r-2D_2,t]$; in particular, $d(o,y)$ cannot be much
less than $t-r$. Thus the distance from the dot representing $\b_2(t)$ to
the boundary arc of $S(x,r;n,t)$ closest to $o$ in the diagram would
typically be slightly larger than $r$. On the other hand, the diameter of
$S(x,r;n,t)\cap\{z\in X\mid d(z,o)=t\}$ is typically larger than this since
we require only that $d(\b_2(t),\b_2'(t))<r$. By our definition, it is not
at first obvious that this puts any upper bound on how far other
representatives of $[\b]$ and $[\b']$ might be from each other. However it
follows from \eqref{E:equiv est} below that if we redefined $\b,\b'$ to be
other members of these respective equivalence classes of generalized
bouquets, then we would still have have $d(\b_2(t),\b_2'(t))<r+10C+8$, which
justifies the fact that $S(x,r;2,t)$ is bounded roughly by rays from the
origin whose distance from $\b_2(t)$ is larger than $r$, but not larger than
$r+10C+8$.

\begin{figure}[ht!]
\begin{tikzpicture}[scale=0.6]
\xdefinecolor{brown}{RGB}{80, 80, 0}
\xdefinecolor{dRed}{RGB}{160, 0, 0}
\clip (-5,-7.0) rectangle (16.0,7.1);
\coordinate (x)  at   (-2:1.7);  \coordinate (o)  at   (0,0);
\coordinate (a)  at    (19:2);   \coordinate (A)  at   (6:4);
\coordinate (b)  at     (0:2);   \coordinate (B)  at (-12:6);
\coordinate (c)  at   (-20:2);   \coordinate (C)  at (-5:10);
\coordinate (am) at (-30:1.4);   \coordinate (Ap) at (28:25);
\shadedraw[thin,top color=red!10, bottom color=red!10] (o) circle (2)
   node[above=1.2cm] {$\pmb{B(o,t)}$};
\shadedraw[blue, thick, top color=blue!10, bottom color=blue!10]
  (am) arc (-30:28:1.4) -- (Ap) arc (28:-30:25.0) -- cycle;
\draw[thin] (o) circle (2);
\node at (20:9.0) {$\pmb{S(x,r;2,t)}$};
\draw[thick,brown,decorate,decoration={snake,amplitude=1.0mm}] (o) -- (a)
-- node[above left, above right=2pt]{$\pmb{\b_1}$} (A);
\draw[thick,decorate,decoration={snake,amplitude=.5mm}] (o) --
  (b) -- node[above right]{$\pmb{\b_2}$} (B);
\draw[thick,dRed,decorate,decoration={snake,amplitude=.2mm}] (o) --
  (c) -- node[above left, near end]{$\pmb{\b_3}$} (C);
\fill[black] (x) circle (3.0pt);
\fill[black] (o) circle (3.0pt) node[left] {$\pmb{o}$};

\end{tikzpicture}
\caption{A basic neighborhood of $x\in\bbd X$}\label{f:S nbd}
\end{figure}

Trivially, $S_0(x;n,t)\subset S(x,r;n,t)$ and $S'(x,r;n,t)\subset
S(x,r;n,t)$. The next two lemmas together show that containments in the
reverse direction (with a change of arguments!) are also possible: these
will be crucial to establishing alternative definitions for $\bt$.

\begin{lem}\label{L:SS0}
For all $R>0$ and $n_0\in\N$, there exist $N\in\N$ and $T>0$ such that
$S(x,R;N,T)\subset I(x;n_0)$ for all $x\in\bbd X$.
\end{lem}

\begin{proof}
It suffices to prove the lemma when $R$ is large, so we assume without loss
of generality that $R\ge 1$. Let $s:=(2R+1)L_{n_0}+R+1$. We will show that
the lemma is true for the choice of parameters $(N,T)$ as long as $N$ is so
large that $L_{N-1}\ge s$, and $T\in[s,L_{N-1}]$. Note that $N>n_0+1$
because $L_{N-1}\ge s>L_{n_0}$. Let $\b^1=(\b_n^1)$ be such that
$x=q(\b^1)$.

For $m>n_0$, we apply Lemma~\ref{L:rough cx} to the bouquet inequality
$$ d(\b_N^1(L_N\mino L_m),\b_m^1(L_N\mino L_m))\le 2C+2 $$
to deduce that
\begin{equation}\label{E:1N1m}
d(\b_N^1(t),\b_m^1(t))\le C+1\,, \qquad 0\le t\le L_{N-1}\mino L_{m-1}\,.
\end{equation}
Suppose now that $y\in S(x,R;N,T)$, and let $\b^2=(\b_n^2)$ be a generalized
bouquet such that $y=q(\b^2)$ and
\begin{equation}\label{E:pNT 12}
d(p_{NT}(\b^1),p_{NT}(\b^2))<R.
\end{equation}
We claim that $d(\b_N^1(t),\b_N^2(t))\le C+1$ for $0\le t\le L_{n_0}$.
Assuming this claim and combining it with \eqref{E:1N1m}, we see that for
$m>n_0$ and $0\le t\le L_{n_0}$,
$$
d(\b_m^1(t),\b_N^2(t)) \le d(\b_m^1(t),\b_N^1(t))+d(\b_N^1(t),\b_N^2(t)) \le
2C+2\,.
$$
Thus we can define a new standard bouquet $\b^3$ with $q(\b^3)=x$ by the
equations $\b_k^3=\b_N^2|_{[0,L_k]}$ for all $k\le n_0$ and $\b_k^3 =
\b_{k+1}^1|_{[0,L_k]}$ for all $k>n_0$. Since $\b^3_n=\b^2_n$ for all $n\le
n_0$, if follows that $y\in S_0(x;n,t)$ for all $n\le n_0$, $0\le t\le L_n$,
and the theorem follows.

It remains to justify the claim. Because $T\ge RL_{n_0}$, the claim follows
by applying Lemma~\ref{L:rough cx} to \eqref{E:pNT 12} if $y\in\bbd X$. It
also follows in the same way if $y\in X$ and $\b_N^2|_{[0,T]}$ is a unit
speed path, i.e.~if $L_N'\ge T$, where $L_N'$ is as in Definition~\ref{D:fl
bouquet} for the finite length bouquet $\b^2$.

Suppose therefore that $L_N'<T$. The inequality $d(\b_N^1(T),\b_N^2(T))<R$,
the $1$-shortness of $\b_N^1|_{[0,T]}$, and the triangle inequality together
imply that $L_N'>T-(R+1)$, which in turn implies that
$L_N'\ge(2R+1)L_{n_0}$. Now $\b_N^1|_{[0,L_N']}$ and $\b_N^2|_{[0,L_N']}$
are unit speed paths and $d(\b_N^1(L_N'),\b_N^2(L_N'))<2R+1$. Since
$T-(R+1)\ge (2R+1)L_{n_0}$, the claim follows as before from
Lemma~\ref{L:rough cx}.
\end{proof}

\begin{rem}\label{R:SS0}
We note two aspects of the proof of Lemma~\ref{L:SS0}:
\begin{enumerate}
\item If $(N,T)$ is one particular choice of data for which the proof
    works, then it also works for any $(N',T')$ such that $T'\ge T$ and
    $L_{N'-1}\ge T'$. In particular, $T$ can be taken to be arbitrarily
    large.
\item Suppose $n_0\in\N$ and $R>0$ are fixed. For every sequence $(N_n)$
    of integers and every unbounded sequence $(T_n)$ such that $0<T_n\le
    L_{N_n-1}$, there exists $n\in\N$ such that $S(x,R;N_n,T_n)\subset
    I(x;n_0)$ for all $x\in\bbd X$. For instance, if we write
    $S(x;n):=S(x,1;n,L_{n-1})$, then there exists $n\in\N$ such that
    $S(x;n)\subset I(x;n_0)$.
\end{enumerate}
\end{rem}

\begin{lem}\label{L:SS'} $S_0(x;n,t)\subset S'(x,10C+8;n,t)$ for
all $x\in\bbd X$, $n\in\N$, and $0\le t\le L_n$.
\end{lem}

\begin{proof}
Suppose $\b,\b'\in\gbq$ with $z:=q(\b)=q(\b')\in\bcl X$. The lemma follows
immediately once we show that
\begin{equation}\label{E:equiv est}
d(\b_n(t),\b_m'(t))\le 5C+4\,, \qquad 0\le t\le L_n\mino L_m,\;n,m\in\N\,.
\end{equation}

Suppose first that $z\in X$ and pick $N\in\N$, $N\ge n\maxo m$, such that
$\b_N(L_N)=\b_N'(L_N)=z$. By Lemma~\ref{L:x1x2}, it follows that
$$ d(\b_N(t),\b_N'(t))\le C\,, \qquad\text{for all } 0\le t\le L_N\,. $$
(Note that we cannot use Lemma~\ref{L:rough cx} to get this estimate because
$\b_N$ and $\b_N'$ might not be of equal length.) Thus for $0\le t\le
L_n\mino L_m$ we have
\begin{align*}
d(\b_n(t),\b_m'(t)) &\le
  d(\b_n(t),\b_N(t)) + d(\b_N(t),\b_N'(t)) + d(\b_N'(t),\b_m'(t)) \\
&\le (2C+2) + C + (2C+2) = 5C+4\,.
\end{align*}

For $z\in\bbd X$, we can similarly deduce \eqref{E:equiv est} from the
limiting estimate
$$
\limsup_{N\to\infty} d(\b_N(t),\b_N'(t))\le C\,, \qquad
  0\le t\le L_n\mino L_m\,.
$$
This last estimate follows from Lemma~\ref{L:rough cx} with data
$a_1=a_2=o$, $b_1=\b_N(L_N)$, $b_2=\b_N'(L_N)$, because of the uniformly
boundedness of $d(\b_N(L_N),\b_N'(L_N))$ and the fact that $(L_n\mino
L_m)/L_N\to 0$ as $N\to\infty$.
\end{proof}

In preparation for the next theorem, let us define some sets associated with
any choice of $x\in\bcl X$ and $R\ge 0$. For $x\in \bbd X$, let
\begin{align*}
\cB_0(x) &=  \{S_0(x;n,t)\mid n\in\N,\, 0<t\le L_n\}\,, \\
\cB_{1,R}(x) &=  \{S(x,r;n,t)\mid r>R,\, n\in\N,\, 0<t\le L_n\}\,, \\
\cB_{2,R}(x) &=  \{S'(x,r;n,t)\mid r>R,\, n\in\N,\, 0<t\le L_n\}\,,
\end{align*}
while for $x\in X$, we simply define
$$ \cB_0(x) = \cB_{1,R}(x) = \cB_{2,R}(x) = \{B(x,r)\mid r>0\}\,. $$

\begin{thm}\label{T:bases}
Suppose $X$ is $C$-rCAT(0). Then
\begin{enumerate}
\item For each $R\ge 0$, $\cB_{1,R}(x)$ is a neighborhood basis for
    $(\bcl X,\bt)$ at $x\in \bcl X$, all of whose elements are open.
\item $\cB_0(x)$ is a neighborhood basis at $x\in \bcl X$ for $(\bcl
    X,\bt)$.
\item For each $R\ge 10C+8$, $\cB_{2,R}(x)$ is a neighborhood basis at
    $x\in \bcl X$ for $(\bcl X,\bt)$.
\end{enumerate}
Also $\bcl X$ is a first countable bordification of $X$.
\end{thm}

\begin{proof}
By definition,
$$ \cB(x):= \{S(x,r;n,t)\mid r>0,\, n\in\N,\, 0<t\le L_n\} $$
forms a neighborhood sub-basis for $\bcl X$ at each $x\in\bcl X$, and all
elements of $\cB(x)$ are open. When $x\in\bbd X$, $\cB(x)=\cB_{1,0}(x)$, so
this is a neighborhood sub-basis, and in fact a neighborhood basis by
Lemma~\ref{L:SS0}. Applying Lemma~\ref{L:SS0} again it is readily deduced
that for all $R>0$, $\cB_{1,R}(x)$ and $\cB_0(x)$ are neighborhood bases at
$x\in\bbd X$. Lemma~\ref{L:SS'} then implies that $\cB_{2,R}(x)$ is a
neighborhood basis at $x\in\bbd X$ whenever $R\ge 10C+8$.

Suppose instead that $x\in X$. We must show that $\cB_0(x)=\{B(x,r)\mid
r>0\}$ is a basis of open neighborhoods at $x\in X$ for $\bt$, or
equivalently that $(\bt)_X$ coincides with the metric topology $\tau$. From
the form of $\cB_0(x)$, it clearly suffices to show that it is a
neighborhood sub-basis at $x\in X$ for $\bt$.

We claim that any ball $B(x,r)$ equals $S(x,r;n,t)$ for any choice of
$t>d(o,x)+r+1$ and $n$ such that $L_n\ge t$. Since $\cB(x)$ is a
neighborhood sub-basis for $\bt$ at $x\in X$, it follows from this claim
that $(\bt)_X$ is at least as fine as $\tau$. To justify our claim, we
suppose that $y\in S(x,r;n,t)$ with $x=q(\b)$, $y=q(\b')$, and
$d(p_{nt}(\b),p_{nt}(\b'))<r$. Because $t>d(o,x)+1$, we must have
$p_{nt}(\b)=x$. Thus
$$ d(o,p_{nt}(\b')) \le d(o,x)+d(x,p_{nt}(\b')) < d(o,x)+r\,. $$
Since $\b_n'|_{[0,L_n']}$ is $1$-short and $t>(d(o,x)+r)+1$, we must have
$p_{nt}(\b')=y$, and so $d(x,y)<r$. Thus $S(x,r;n,t)\subset B(x,r)$. The
reverse containment is proved in a similar fashion.

To prove that conversely $\tau$ is at least as fine as $(\bt)_X$, we show
that for fixed but arbitrary $x\in X$, $0<r<1$, $n\in\N$, and $0<t\le L_n$,
$S(x,r;n,t)$ contains some ball $B(x,\d)$. First pick a $(D/2)$-short unit
speed path of length $L$ from $o$ to $x$, and then let $\b=(\b_k)\in\gbq$ be
its child. Also let $\d:=r\mino[D(d(o,x))/4]$. For $y\in B(x,\d)$, we pick a
unit speed path $\la_y:[0,l_y]\to X$ from $x$ to $y$, with $l_y<\d$, and
then define $\g^y:[0,L+l_y]\to X$ by the formula
$$
\g^y(s) =
\begin{cases}
  \g(s), & 0\le s\le L, \\
  \la_y(s-L), & L<s\le L+l_y.
\end{cases}
$$
Using the $1$-Lipschitz property of $D$, it is readily verified that $\g^y$
is $D$-short, and so $\g^y$ is a mother bouquet from $o$ to $y$. By
construction it is clear that $d(\g(s),\g^y(s))<r$ for all $s>0$. In
particular this last inequality holds for $s=t$, and so $B(x,\d)\subset
S(x,r;n,t)$.

As for first countability, it is clear that there exists a countable
neighborhood base at each $x\in X$, and  a countable neighborhood base at
$x\in\bbd X$ is given by Remark~\ref{R:SS0}(b). To see that $\bcl X$ is a
bordification of $X$, we need to show that the basic neighborhood
$S_0(x;n,t)$ of $x\in\bbd X$ always contains a point of $X$. But this is
easy since $\b_n(L_n)\in S_0(x;n,t)\cap X$ whenever $\b=(\b_n)\in
q^{-1}(x)$.
\end{proof}

We already know that $\bcl X$ {\it as a set} is independent of the basepoint
$o$ (Corollary~\ref{C:std asymp}). We now show that the associated topology
is also independent of $o$.

\begin{thm}\label{T:top-oo'}
The topology $\bt$ is independent of the basepoint $o$.
\end{thm}

\begin{proof}
Suppose $o,o'\in X$ are two basepoints in $X$. In view of the definition of
the neighborhood bases given in Theorem~\ref{T:bases}, it suffices to show
that the topology with respect to these two basepoints is the same in the
vicinity of each $x\in\bbd X$. By symmetry of $o,o'$, it therefore suffices
to exhibit a neighborhood basis at $x$ with basepoint $o'$ such that every
element of this basis contains some neighborhood of $x$ with respect to the
basepoint $o$. To facilitate this comparison, we add $\omega\in\{o,o'\}$ as
a superscript to our notation, writing $\bt^\omega$, $S^\omega(x,r;n,t)$,
$\cB_{1,R}^\omega(x)$, etc. In view of Theorem~\ref{T:bouquet-oo'}, we can
write $\bbd X$ and $\bcl X$ without such a superscript; however we should
write the quotient maps as $q^\omega:\Gbq{\omega}\to\bcl X$.

By Theorem~\ref{T:bases}, $\cB_{1,R}^{o'}(x)$ is neighborhood basis for
$(\bcl X,\bt^{o'})$ for any given $R\ge 0$. Choosing $R:=6C+4+2d(o,o')$, a
general element of $\cB_{1,R}^{o'}(x)$ has the form $S^{o'}(x,r';n,t)$ for
some $r'>R$, $n\in\N$, and $0<t\le L_n$. We claim that such a general basis
element contains the neighborhood $S^o(x,r;n,t)$, where $r=r'-R$.

Suppose $y\in S^{o}(x,r;n,t)$ for some $r>0$, $n\in\N$, $0<t\le L_n$. Thus
for $z\in\{x,y\}$, there are generalized bouquets $\b^{z,1} =
(\b_m^{z,1})_{m=1}^\infty$ from $o$ such that $q^o(\b^{z,1})=z$ and
$d(\b_n^{x,1}(t), \b_n^{y,1}(t))<r$.

Let us first examine the construction in the proof of
Theorem~\ref{T:bouquet-oo'}, where $c=c'=2C+2$ and we reserve $r,n,t$ to
have their specific meanings in the context of $S^{o}(x,r;n,t)$. Stripped of
its fine details, the construction of a standard bouquet
$\b^{x,2}=(\b_m^{x,2})_{m=1}^\infty$ from $o'$ that is asymptotic to
$\b^{x,1}$ in the proof of Theorem~\ref{T:bouquet-oo'} is as follows: first
we take a sequence of sufficiently short paths whose $m$th entry is a path
from $o'$ to $x_m:=\b_m^{x,1}(L_m)$ (we can use $x_m$ as the final point
rather than some intermediate point $y_m$ as in the original proof because
in this section $D(t):=1/(1\maxo(2t))$: see Remark~\ref{R:bouquet-oo'}),
then we take a subsequence of this sequence, and finally we suitably prune
this subsequence. In particular there exists $N\ge n$ such that $\b_n^{x,2}$
is an initial segment of a sufficiently short path $\la_N:[0,M_N]\to X$ from
$o'$ to $x_N$ that is parametrized by arclength. Writing $u=\b_n^{x,1}(t)$,
$v=\b_N^{x,1}(t)$, we have $d(u,v)\le 2C+2$ by one of the defining
conditions for standard bouquets. Next letting $w:=\b_N^{x,2}(t)$, it
follows from Remark~\ref{R:o1o2} that $d(v,w)\le C+d(o,o')$, and so
$d(u,w)\le 3C+2+d(o,o')$.

Suppose $y\in\bbd X$. Applying the argument of the previous paragraph to $y$
in place of $x$, we get a standard bouquet
$\b^{y,2}=(\b_m^{y,2})_{m=1}^\infty$ from $o'$ that is asymptotic to
$\b^{y,1}$ such that $d(u',w')\le 3C+2+d(o,o')$ where $u'=\b_n^{y,1}(t)$ and
$w':=\b_N^{y,2}(t)$. Since $d(u,u')<r$, we conclude that
$$ d(w,w') \le d(w,u) + d(u,u') + d(u',w') < r+R\,, $$
and so $y\in S^{o'}(x,r+R;n,t)$ as claimed.

Suppose instead that $y\in X$. Let $\b^{y,2}$ be the child of a mother
bouquet from $o'$ to $y$ of length $L$, let $L'$ be length of $\b_n^{y,1}$,
and let $u'=\b_n^{y,1}(t)$, $w':=\b_N^{y,2}(t)$. By Remark~\ref{R:o1o2}, we
see that $d(u',w')\le C+d(o,o')$ if $t<L\mino L'$, and otherwise shortness
gives $d(u',w')\le 1+d(o,o')$. As before
$$ d(w,w') \le d(w,u)+d(u,u')+d(u',w') < r+4C+3+2d(o,o') < r+R\,, $$
and so again $y\in S^{o'}(x,r+R;n,t)$. Thus our claim follows and the proof
is done.
\end{proof}

We are now ready to prove Theorem~\ref{T:Hausdorff intro}.

\begin{proof}[Proof of Theorem~\ref{T:Hausdorff intro}]
First countability was proved in Theorem~\ref{T:bases}. To prove that $\bcl
X$ is Hausdorff, we suppose $x,y\in\bcl X$ are distinct. If one or both of
$x,y$ lie in $X$, then Theorem~\ref{T:bases} implies that they have disjoint
neighborhoods: for instance if $x\in\bbd X$ and $y\in X$ then $B(y,1)$ and
$S(x,1,n,t)$ are disjoint whenever $t>|y|+2$ and $L_n\ge t$ (as in the proof
that $(\bt)_X$ is at least as fine as $\tau$ in the proof of
Theorem~\ref{T:bases}). It therefore suffices to consider the case where
$x,y\in\bbd X$, and so $x=q(\b^x)$, $y=q(\b^y)$, where $\b^x,\b^y$ are
non-asymptotic standard bouquets from $o$. Since $\b^x,\b^y$ are not
asymptotic, we can find $n\in\N$ so large that
$d(\b_n^x(L_n),\b_n^y(L_n))\ge 15C+14$. Letting $U=S(x,1;n,L_n)$ and
$V=S(y,1;n,L_n)$, it follows readily from \eqref{E:equiv est} that $U$ and
$V$ are disjoint neighborhoods of $x$ and $y$ in $\bcl X$, and so $\bcl X$
is Hausdorff.

We claim that $\gbq$ is a closed subset of $P$. Convergence in the product
space $P$ corresponds to pointwise convergence in $\gbq$ (meaning
convergence for each choice of $n,t$), so justifying this claim requires us
to show that a pointwise limit of a sequence of generalized bouquets is a
generalized bouquet. The important step is to note that if for some fixed
$n\in\N$ and all $m\in\N$, $\b_{nm}$ is a path of subunit speed and length
at most $L_n$ from $o$ to $x_m$ (where $L_n$ is defined as always for
generalized bouquets), and if $\b_{nm}(t)$ is pointwise convergent for all
$0\le t\le L_n$, then each of these paths lies in the metric space $X$, so
we may apply the Arzel\`a-Ascoli theorem to deduce that $\b_{nm}$ converge
uniformly to some limiting path $\b_n:[0,L_n]\to X$ of subunit speed. Since
the short function $D$ is continuous, $\b_n$ is $D$-short if each $\b_{nm}$
is $D$-short. It readily follows that a pointwise convergent sequence of
mother bouquets converges uniformly to a mother bouquet, and that a
pointwise convergent sequence of standard bouquets converges to a standard
bouquet. The claim follows.

Suppose next that $X$ is proper. Then $P$ is a product of compact spaces and
so compact. Compactness is inherited by closed subspaces and by quotients
so, applying the above claim, we see that $\bcl X$ is compact. Using
Theorem~\ref{T:bases}, we see that $X=\bigcup_{x\in X} B(x,1)$ is open in
$\bcl X$, and so $\bbd X$ is closed in $\bcl X$. Thus $\bbd X$ is also
compact.
\end{proof}

\begin{proof}[Proof of Theorem~\ref{T:CAT0 main}]
By Proposition~\ref{P:cat0 is rcat0}, $X$ is $C$-rCAT(0) for $C:=2+\sqrt 3$.
By Theorem~\ref{T:CAT0 bij}, we can identify $\icl X$ and $\bcl X$ as sets.
The neighborhood bases for the bouquet topology $\bt$ given in
Theorem~\ref{T:bases} and for the cone topology $\ct$ in
Definition~\ref{D:cone top} coincide at each $x\in X$, so it suffices to
consider the two neighborhood bases at points $x\in \ibd X=\bbd X$.

According to Theorem~\ref{T:geod ray}, there exists a (unique) unit speed
geodesic ray $\g^z:[0,\infty)\to X$ from $o$ that is asymptotic to any given
standard bouquet $\b\in q^{-1}(z)$. We view $\g^z$ as an element of $\gbq$
by identifying it with the standard bouquet $\b^z=(\b_n^z)$, where
$\b_n^z:=\g^z|_{[0,L_n]}$. For $z\in X$, let $\g^z$ be the unique unit speed
geodesic segment $\g^z:[0,d(o,z)]\to X$ from $o$ to $z$, and identify $\g^z$
with its child $\b^z=(\b_n^z)$. In this way the set of these (unique) unit
speed segments or rays from $o$ to all $z\in\bcl X$ is identified with a
subset $\Gbq{*}$ of $\gbq$ and $q':=q|_{\Gbq{*}}:\Gbq{*}\to\bcl X$ is
bijective, so we identify $\bcl X=\icl X$ with $\Gbq{*}$. Note that
$p_{nt}((q')^{-1}(z))$ is independent of $n$: in fact it equals $\g^z(t)$
(or simply $z$ if $z\in X$ and $t>d(o,z)$).

Viewing $\ibd X$ in this manner, it follows from Definition~\ref{D:cone top}
that the basic neighborhood $U(x,r,t)$ for the cone topology $\ct$ at
$x\in\ibd X$ is contained in $S(x,r;n,t)\in\cB_{1,0}(x)$. On the other hand,
it follows from \eqref{E:equiv est} that $S(x,r;n,t)\subset U(x,r+10C+8,t)$.
But the collection of sets $U(x,r',t)$ for all $t>0$ and $r'>10C+8$ forms an
open basis for $\ct$ at $x$: this follows readily from the containment
$$ U(x,10C+9,t(10C+9)/r)\subset U(x,r,t)\,,\qquad 0<r<1,\;\; 0<t\,, $$
which in turn follows from the CAT(0) condition.
\end{proof}

\begin{proof}[Proof of Theorem~\ref{T:Gromov main}]
It suffices to compare the neighborhood bases at $x\in\Gbd X$. We assume
that $\d>0$ is such that $X$ is $\d$-hyperbolic, and so $X$ is also
$C$-rCAT(0) for $C:=2+4\d$ by Proposition~\ref{P:Gh is rcat0}. The
identification of a standard bouquet with the Gromov sequence of its tips
induces an identification of $\bcl X$ and $\Gcl X$ as sets; see
Remark~\ref{R:tip} and Theorem~\ref{T:Gromov bij}.

We take as a $\Gt$-neighborhood basis at $x$ the standard one given by
Definition~\ref{D:Gromov top}, namely $\{V(x,R)\mid R>0\}$. Fixing $R$, we
claim that
\begin{align*}
S(x,1;N,t)\subset V(x,R) \quad \text{ whenever } &N\in\N,\; t>0
  \text{ are so large that } \\
  &\qquad L_N\ge t>R+(4C+7)/2\,.
\end{align*}
To prove this claim, we assume that $y\in S(x,1;N,t)$ for such a choice of
$N$ and $t$, and separately show that $y\in V(x,R)$ when $y\in \bbd X$ and
when $y\in X$.

Suppose first that $y\in\bbd X$. We take the Gromov sequences of tips
$(x_n)$ of $\b^x=(\b_n^x)$, and $(y_n)$ of $\b^y=(\b_n^y)$, where
$\b^x,\b^y$ are standard bouquets such that $q(\b^x)=x$, $q(\b^y)=y$, and
$d(p_{Nt}(\b^x),p_{Nt}(\b^y))<1$. For every $m,n\ge N$, we have
\begin{align*}
d(x_m,y_n) &= d(\b_m^x(L_m),\b_n^y(L_n)) \\
  &\le d(\b_m^x(L_m),\b_m^x(t)) + d(\b_m^x(t),\b_N^x(t))+ \\
  &\qquad + d(\b_N^x(t),\b_N^y(t)) + d(\b_N^y(t),\b_n^y(t)) +
    d(\b_n^y(t),\b_n^y(L_n)) \\
  &\le (L_m-t)+(2C+2)+1+(2C+2)+(L_n-t) \\
  &= (L_m+L_n-2t) + 4C+5.
\end{align*}
But $d(o,x_m)\ge L_m-1$ and $d(o,y_n)\ge L_n-1$, so
$$ \IP<x_m,y_n;o> \ge t-(4C+7)/2 > R\,. $$
Thus $S(x,1;N,t)\cap\bbd X\subset V(x,R)$.

The proof for $y\in X$ is mostly similar, so we mention only the
differences. First, let $y_n=\b_n^y(L_n)$, where $\b^y$ is a finite length
bouquet from $o$ to $y$. Because $d(\b_N^x(t),\b_N^y)<1$ and $\b_N^x$ is
$1$-short, it follows that $L_n'\ge t-2$ for $n\ge N$. Thus either $L_n'\ge
t$ and we deduce as before that $d(x_m,y_n) \le (L_m+L_n'-2t) + 4C+5$, or
$L_n'<t$. In the latter case we have $\b_n^y(t)=\b_n^y(L_n')=y$ for $n\ge N$
and so
\begin{align*}
d(x_m,y_n) &= d(\b_m^x(L_m),\b_n^y(L_n')) \\
  &\le d(\b_m^x(L_m),\b_m^x(t)) + d(\b_m^x(t),\b_N^x(t))+ \\
  &\qquad + d(\b_N^x(t),\b_N^y(t)) + d(\b_N^y(t),\b_n^y(L_n')) \\
  &\le (L_m-t)+(2C+2)+1+0 \\
  &= (L_m-t) + 2C + 3 \le (L_m+L_n'-2t+2) + 2C + 3
\end{align*}
and so we again have $d(x_m,y_n) \le (L_m+L_n'-2t) + 4C+5$. Since also
$d(o,y_n)\ge L_n'-1$, we deduce as before that
$$ \IP<x_m,y_n;o> \ge t-(4C+7)/2 > R\,. $$
But $y_n=y$ for all sufficiently large $n$, so
$$ \IP<x_m,y;o> \ge t-(4C+7)/2 > R\,. $$
Hence $S(x,1;N,t)\cap X\subset V(x,R)$. Thus the claim follows and so $\bt$
is finer than $\Gt$.

It remains to show conversely that $\Gt$ is finer than $\bt$. To show this,
we take $\{S(x,r;n,L_{n-1})\mid n\in\N,\,r>4\d+2\}$ as a neighborhood basis
for $\bt$ at $x\in\bbd X$; see Theorem~\ref{T:bases} and
Remark~\ref{R:SS0}(b). It suffices to show that
$$
V(x,R)\subset S(x,r;n,L_{n-1})\,, \qquad
  n\in\N,\; R\ge L_{n-1},\; r>4\d+2\,.
$$

Consider first $y\in V(x,R)\cap\bbd X$. By eliminating some initial elements
if necessary from the sequences given by Definition~\ref{D:Gromov top}, we
may assume that $(a_j)$ and $(b_j)$ are Gromov sequences such that
$[(a_j)]=x$, $[(b_j)]=y$, $\IP<a_j,b_k;o>\ge R$ for all $j,k\in\N$, and
$$
\IP<a_j,a_k;o>\mino \IP<b_j,b_k;o> \ge L_j\mino L_k\,, \qquad j,k\in\N\,.
$$
The last inequality implies in particular that $d(o,a_j)\mino d(o,b_j)\ge
L_j$, $j\in\N$.

For each $j\in\N$, let $\b_j^x,\b_j^y$ be the initial segments of length
$L_j$ of $D$-short paths from $o$ to $a_j,b_j$, respectively, and let
$a_j'=\b_j^x(L_j)$, $b_j'=\b_j^y(L_j)$ be the associated tips. The Tripod
Lemma (Lemma~\ref{L:Tripod}) tells us that $\b_j^x$ and $\b_j^y$ are
$(4\d+2,D)$-bouquets and that $d(\b_j^x(s),\b_k^y(s))\le 4\d+2$ for all
$j,k\in\N$, $s\le R\mino L_j\mino L_k$. This last estimate remains true
after taking pruned subsequences of $(\b_j^x)$ and $(\b_j^y)$, which we do
if necessary in order to get standard bouquets. We may thus assume without
loss of generality that $\b^x,\b^y$ are standard bouquets satisfying
$q(\b^x)=x$, $q(\b^y)=y$, and $d(\b_j^x(s),\b_k^y(s))\le 4\d+2$ for all
$s\in[0,R\mino L_j\mino L_k]$ and $j,k\in\N$. In particular,
$d(\b_n^x(t),\b_n^y(t))\le 4\d+2$ for $t=L_{n-1}$, and so $y\in
S(x,r;n,L_{n-1})$ for every $r>4\d+2$.

The analysis for $y\in V(x,R)\cap X$ is fairly similar. First we choose a
Gromov sequence $(a_n)$ such that $[(a_j)]=x$, $\IP<a_j,y;o>\ge R$, and
$\IP<a_j,a_k;o>\ge L_j\mino L_k$ for all $j,k\in\N$. In particular,
$d(o,y)\ge R$ and $d(o,a_j)\ge L_j$ for all $j\in\N$. Define $(\b_j^x)$ as
before, and let $\b^y=(\b_j^y)_{j=1}^\infty$ be a finite length bouquet from
$o$ to $y$. Then $(\b_x^n)$ is a $(4\d+2,D)$-bouquet and
$d(\b_j^x(s),\b_k^y(s))\le 4\d+2$ for all $s\in[0,R\mino L_j\mino L_k]$. In
particular, $y\in S(x,r;n,L_{n-1})$ for every $r>4\d+2$.
\end{proof}


\end{document}